\documentclass[reqno,11pt]{amsart}

\usepackage[a4paper,
left=2.2cm,right=2.2cm,
top=2.3cm,bottom=2.3cm]{geometry}

\usepackage{graphicx}
\usepackage{amsmath,amssymb,amsfonts,amsthm,amscd,bm}
\usepackage{mathrsfs}
\usepackage{array}
\usepackage{enumerate}
\usepackage{xcolor}
\usepackage{textcomp}
\usepackage[colorlinks=true,linkcolor=blue,citecolor=blue,urlcolor=blue]{hyperref}
\usepackage{float}
\usepackage{subcaption}
\usepackage{longtable}
\usepackage{rotating}
\usepackage{booktabs}
\usepackage{adjustbox}
\usepackage{threeparttable}
\usepackage{multirow}
\usepackage{listings}

\usepackage[ruled,vlined]{algorithm2e}

\DeclareMathOperator{\supp}{supp}

\theoremstyle{plain}
\newtheorem{theorem}{Theorem}[section]
\newtheorem{lemma}[theorem]{Lemma}
\newtheorem{proposition}[theorem]{Proposition}
\newtheorem{corollary}[theorem]{Corollary}

\theoremstyle{definition}
\newtheorem{assumption}[theorem]{Assumption}

\theoremstyle{remark}
\newtheorem{remark}[theorem]{Remark}

\graphicspath{{figs/}}

\title[Affine Normal Directions via Log-Determinant Geometry]
{Affine Normal Directions via Log-Determinant Geometry:\\
Scalable Computation under Sparse Polynomial Structure}
\author[Y.S. Niu et al.]{Yi-Shuai Niu$^{1}$, Artan Sheshmani$^{1,3}$, and Shing-Tung Yau$^{1,2}$}



\date{\today}

\begin{document}

\begin{abstract}
Affine normal directions provide intrinsic affine-invariant descent directions derived from the geometry of level sets. Their practical use, however, has long been hindered by the need to evaluate third-order derivatives and invert tangent Hessians, which becomes computationally prohibitive in high dimensions. In this paper, we show that affine normal computation admits an exact reduction to second-order structure: the classical third-order contraction term is precisely the gradient of the log-determinant of the tangent Hessian. This identity replaces explicit third-order tensor contraction by a matrix-free formulation based on tangent linear solves, Hessian--vector products, and log-determinant gradient evaluation.
Building on this reduction, we develop exact and stochastic matrix-free procedures for affine normal evaluation. For sparse polynomial objectives, the algebraic closure of derivatives further yields efficient sparse kernels for gradients, Hessian--vector products, and directional third-order contractions, leading to scalable implementations whose cost is governed by the sparsity structure of the polynomial representation. We establish end-to-end complexity bounds showing near-linear scaling with respect to the relevant sparsity scale under fixed stochastic and Krylov budgets.
Numerical experiments confirm that the proposed MF-LogDet formulation reproduces the original autodifferentiation-based affine normal direction to near machine precision, delivers substantial runtime improvements in moderate and high dimensions, and exhibits empirical near-linear scaling in both dimension and sparsity. These results provide a practical computational route for affine normal evaluation and reveal a new connection between affine differential geometry, log-determinant curvature, and large-scale structured optimization.
\end{abstract}

\maketitle
\tableofcontents

\section{Introduction}

Affine invariance has long played an important role in optimization. Classical second-order methods, most notably Newton's method, are invariant under affine changes of variables and therefore adapt naturally to local curvature. More recently, affine differential geometry has suggested a different route to affine-invariant optimization through the \emph{affine normal} direction of level sets. Unlike the Euclidean normal, namely the gradient direction, the affine normal is covariant under volume-preserving affine transformations and reflects both intrinsic affine curvature and local volume distortion of level sets. This makes it a natural geometric object for constructing optimization methods that are not tied to a particular Euclidean coordinate system.

In the work of Cheng--Cheng--Yau~\cite{cheng2005}, the affine normal was introduced into unconstrained convex optimization as a descent direction derived from level-set geometry. That work established the geometric idea of descending along affine normal directions, but did not provide a complete algorithmic framework, a convergence theory, or a computational treatment suitable for high-dimensional settings. Moreover, outside locally elliptic regions, the affine normal suffers from orientation ambiguity, and its behavior in nonconvex problems was not systematically understood.

To address these issues, our recent work~\cite{AND2026} introduced \textbf{\emph{Yau's affine normal descent} (YAND)}, a unified affine-invariant framework for smooth optimization in both convex and nonconvex regimes. In elliptic\footnote{A point $z$ is called \textbf{elliptic} if the tangent--tangent Hessian at $z$ is positive definite; It is called \textbf{hyperbolic} if that block is indefinite, and \textbf{parabolic} (or degenerate)  if it is singular.} regions, YAND agrees with the classical affine normal, while in hyperbolic and parabolic regions it resolves the orientation ambiguity by enforcing consistency with the Euclidean gradient direction. That work established convergence guarantees under several structural assumptions, including global linear convergence under strong convexity and the Polyak--{\L}ojasiewicz condition, as well as local superlinear convergence under affine second-order consistency. The present paper complements~\cite{AND2026} by addressing the main computational bottleneck left open there: the efficient evaluation of affine normal directions in high-dimensional structured problems.

The computational difficulty is already visible in the classical local formula. In coordinates aligned with the gradient, the affine normal admits the representation
\begin{equation}\label{eq:intro-AN}
d_{\mathrm{AN}}(z)\ \propto\
\begin{pmatrix}
\displaystyle f^{ij}(z)\Big(f_{n+1,j}(z)-\frac{1}{n+2}\,\|\nabla f(z)\|\,f^{pq}(z)f_{pqj}(z)\Big)\\[4pt]
-1
\end{pmatrix},
\end{equation}
where $H_T(z):=[f_{ij}(z)]_{i,j=1}^n$ is the tangent Hessian block and $[f^{ij}(z)]=H_T(z)^{-1}$. Formula~\eqref{eq:intro-AN} reveals two intrinsic computational bottlenecks: repeated tangent linear solves involving $H_T(z)$, and the evaluation of the third-order contraction
\[
a_j(z):=f^{pq}(z)f_{pqj}(z).
\]
In high dimensions, these operations are expensive, and for generic smooth objectives they offer little reusable structure from one iterate to the next. This computational burden has so far limited the practical scalability of affine-normal-based (cf. Yau-like or Yau-type \cite{AND2026}) optimization methods.

The central observation of this paper is that the apparent third-order complexity of affine normal computation can be reduced to second-order structure. More precisely, we show that
\[
f^{pq}f_{pqi}
=
\partial_{i}\log\det(H_T),
\]
so that the critical affine-normal contraction can be identified with the gradient of the log-determinant of the tangent Hessian. Although this identity follows formally from Jacobi's formula, it has substantial algorithmic consequences: it replaces explicit third-order tensor contraction by a formulation based on tangent Hessian operators, linear system solves, and log-determinant gradient evaluation. This yields a matrix-free route to affine normal computation.

This reduction becomes especially powerful in polynomial optimization. For polynomial objectives, including high-degree polynomial minimization and optimization with polynomial surrogate models, all derivatives remain within the same sparse monomial family. This algebraic closure implies that Hessian--vector products can be computed efficiently by sparse monomial contractions, without explicitly forming the Hessian matrix. As a consequence, both the tangent operator $H_Tv$ and the reduced affine-curvature term $\nabla_t \log\det(H_T)$ can be evaluated in a matrix-free manner with cost comparable to Hessian--vector multiplication. Combined with stochastic trace estimation and Krylov linear solvers, this opens the door to scalable affine normal computation for high-dimensional sparse polynomial objectives.

The main contributions of this paper are as follows.
\begin{itemize}
\item We identify the classical affine-normal contraction term with the gradient of the log-determinant of the tangent Hessian, thereby reducing affine normal computation from explicit third-order contraction to second-order structure.

\item We derive exact and stochastic matrix-free (MF-LogDet) procedures for affine normal evaluation based on tangent linear solves, log-determinant gradient evaluation, stochastic trace estimation, and directional third-order kernels, without requiring explicit construction of the tangent Hessian inverse or any full third-order tensor.

\item For sparse polynomial objectives, we develop structured sparse kernels for gradients, Hessian--vector products, and directional third-order contractions, yielding an efficient matrix-free implementation of the proposed reformulation.

\item We establish end-to-end complexity bounds showing that, in the sparse polynomial regime, the cost of affine normal evaluation is governed by matrix-free second-order primitives together with the stochastic trace and Krylov budgets, and is near-linear in the relevant sparsity scale.

\item We provide numerical evidence showing that the proposed MF-LogDet formulation reproduces the explicit autodifferentiation-based affine normal direction to near machine precision, yields substantial runtime improvements in moderate and high dimensions, and exhibits empirical near-linear scaling with respect to both dimension and sparsity.
\end{itemize}

Conceptually, these results show that affine normal computation can be transformed from a seemingly high-order geometric task into a matrix-free second-order procedure. Algorithmically, they provide a practical route for deploying affine-normal-based methods in high-dimensional sparse polynomial optimization.

While the resulting matrix-free formulation places affine-normal computation within the same complexity regime as Hessian--vector-based methods, it is important to clarify its relationship with classical Newton-type approaches. Compared with matrix-free Newton--CG, the proposed MF-LogDet affine-normal computation generally incurs an additional stochastic-trace overhead. Hence, from a purely per-iteration viewpoint, Newton--CG may remain cheaper. The significance of the present work is of a different nature: it shows that affine-normal directions, previously viewed as computationally
prohibitive due to their implicit third-order structure, can in fact be evaluated using only matrix-free second-order primitives,
up to stochastic and Krylov factors. Whether this additional geometric information translates into improved practical performance depends on the problem structure, conditioning, and outer-iteration behavior, and is therefore not the primary focus
of the present work.

The remainder of the paper is organized as follows.
Section~\ref{sec:related} reviews related work.
Section~\ref{sec:analytic} recalls the analytic affine normal formula and rewrites it in a form suitable for computation.
Section~\ref{sec:logdet} establishes the key log-determinant identity that eliminates explicit third-order contractions.
Section~\ref{sec:poly} develops sparse polynomial derivative formulas and matrix-free operator kernels.
Section~\ref{sec:stochastic} introduces stochastic trace estimation and Krylov-based linear solves for the log-determinant gradient subproblem.
Section~\ref{sec:algorithm} presents the MF-LogDet computational framework.
Section~\ref{sec:complexity} provides the resulting complexity analysis.
Section~\ref{sec:experiments} reports numerical experiments.
Finally, Section~\ref{sec:conclusion} concludes the paper.

\section{Related Work}
\label{sec:related}

This work is related to affine-invariant optimization, second-order methods, stochastic numerical linear algebra, and structured computation for polynomial objectives.

\subsection{Affine-invariant methods}

Affine-invariant optimization has been studied from geometric perspectives, particularly within affine differential geometry, which provides intrinsic notions of curvature and normal directions invariant under volume-preserving affine transformations; see, for example, \cite{nomizu1994affine,li1993affine}. The affine normal is one such object and encodes geometric information of level sets beyond the Euclidean setting.

Cheng--Cheng--Yau~\cite{cheng2005} introduced affine normal directions into convex optimization from an affine differential-geometric viewpoint, but did not address theoretical analysis, scalable computation, or nonconvex settings. Our recent work~\cite{AND2026} developed the YAND framework and established its convergence theory in both convex and nonconvex regimes. The present paper complements that work by addressing the computational question of how affine normal directions can be evaluated efficiently in high-dimensional settings.

\subsection{Second-order methods}

Classical second-order methods, most notably Newton's method, achieve affine invariance through explicit use of the Hessian; see \cite{nocedal2006numerical}. In high dimensions, however, explicit formation or inversion of the Hessian is often computationally prohibitive. This has motivated matrix-free variants, such as Newton-CG and truncated Newton methods, which rely on Hessian--vector products and Krylov subspace solvers; see \cite{dembo1982inexact,nocedal2006numerical,saad2003}.

\subsection{Natural gradient and information geometry}

Geometric invariance also plays a central role in information geometry, most notably through Amari's natural gradient method \cite{amari1998natural,amari2016information}. Natural gradient is based on an intrinsic metric on parameter space, typically induced by the Fisher information matrix. In contrast, the affine normal considered here is derived from the extrinsic geometry of level sets. The two viewpoints are therefore conceptually related, but mathematically distinct.

\subsection{Stochastic trace and log-determinant estimation}

Large-scale trace and log-determinant estimation has been extensively studied in numerical linear algebra and statistics; see, for example, \cite{girard1989trace,hutchinson1990trace,avron2011randomized,bai1996some,han2015large,ubaru2017fast}. A standard approach is to combine randomized trace estimators with Krylov subspace methods in order to evaluate quantities involving inverse operators in a matrix-free manner.

The present work uses these ideas in a different geometric setting. The key affine-normal contraction term is first rewritten as a log-determinant gradient and then evaluated through stochastic trace estimation and linear system solves. This provides a direct connection between affine differential geometry and modern randomized numerical linear algebra.

\subsection{Polynomial optimization and structured computation}

Polynomial optimization provides a natural setting for the present approach because derivatives of polynomial objectives retain the same algebraic structure; see, for example, \cite{lasserre2001global,parrilo2003semidefinite,lasserre2015introduction}. In sparse polynomial problems, this structure can often be exploited to evaluate gradients, Hessian--vector products, and related directional derivative quantities efficiently \cite{waki2006sums}.

The present work shows that, in this setting, affine normal computation can be reduced to matrix-free primitives whose cost is governed by the sparsity structure of the polynomial representation. To the best of our knowledge, this is the first work to establish such a computational reduction for affine-normal-based optimization methods.

\section{Analytic Formula and Matrix-Free Reformulation}
\label{sec:analytic}

In this section we recall the classical analytic expression
of the affine normal direction, identify its computational bottlenecks,
and reformulate it in a coordinate-free matrix-free form
that is suitable for large-scale implementation.

\subsection{Coordinate conventions}
\label{sec:coordinates}

Throughout the paper, we distinguish between ambient coordinates
and tangent coordinates.

Let $f:\mathbb{R}^{n+1}\to\mathbb{R}$ be a $C^3$ function and
let $z\in\mathbb{R}^{n+1}$ satisfy $\nabla f(z)\neq 0$.
We consider the level set
\[
M_z = \{x \in \mathbb{R}^{n+1} : f(x)=f(z)\}.
\]
Its tangent space at $z$ is
\[
T_z M_z = \{v \in \mathbb{R}^{n+1} : \nabla f(z)^\top v = 0\}.
\]

We introduce a normal-aligned orthonormal basis
\[
Q = [T \ \ \nu] \in \mathbb{R}^{(n+1)\times(n+1)},
\]
where $\nu=\nabla f(z)/\|\nabla f(z)\|$ is the unit normal direction,
and the columns of $T$ form an orthonormal basis of $T_z M_z$.

In this basis, any point $x$ in a neighborhood of $z$ is represented as
\[
y = Q^\top (x - z),
\]
where the first $n$ coordinates correspond to tangent directions
and the last coordinate corresponds to the normal direction.
We refer to this coordinate system as \emph{normal-aligned coordinates}.

In these coordinates, we adopt the shorthand notation
\[
f_i=\partial_i f,\qquad
f_{ij}=\partial_i\partial_j f,\qquad
f_{pqi}=\partial_p\partial_q\partial_i f,
\]
where indices
\[
i,j,p,q\in\{1,\dots,n\}
\]
always refer to tangent coordinates, while the index $n+1$ refers to the normal direction.
Here $\partial_i$ denotes differentiation along the $i$-th tangent coordinate direction.

Accordingly:
\begin{itemize}
\item Indices $i,j,p,q \in \{1,\dots,n\}$ correspond to tangent directions, while the index $n+1$ corresponds to the normal direction.
\item Ambient quantities such as $\nabla f(z)$ and $\nabla^2 f(z)$ are defined in $\mathbb{R}^{n+1}$. Their tangent counterparts are obtained via orthogonal projection onto the tangent space:
\[
\nabla_t f(z) = T T^\top \nabla f(z),
\qquad
H_T(z) = T^\top \nabla^2 f(z)\, T.
\]
These definitions separate ambient and intrinsic (tangent) geometry, which is essential for the affine-invariant formulation developed below.
\end{itemize}

Under this convention, all tensor contractions appearing in the affine normal formula are taken over tangent indices.

\subsection{Analytic affine normal formula in normal-aligned coordinates}
Let
\[
H_T(z):=[f_{ij}(z)]_{i,j=1}^n
\]
denote the tangent Hessian block,
and let
\[
[f^{ij}(z)] = H_T(z)^{-1}.
\]
The classical analytic affine normal direction can be written,
up to a nonzero scalar multiple, as
\begin{equation}
\label{eq:AN-analytic-sec2}
d_{\mathrm{AN}}(z)\ \propto\
\begin{pmatrix}
\displaystyle
f^{ij}(z)\Big(
f_{n+1,j}(z)
-\frac{1}{n+2}\,\|\nabla f(z)\|\, f^{pq}(z)f_{pqj}(z)
\Big)
\\[6pt]
-1
\end{pmatrix}.
\end{equation}
Throughout this paper we adopt the Einstein summation convention
over repeated tangent indices. That is, whenever an index appears twice
in a term, summation over that index from $1$ to $n$ is implied.
For example,
\[
f^{ij} f_{jk}
\;=\;
\sum_{j=1}^n f^{ij} f_{jk},
\qquad
f^{pq} f_{pqj}
\;=\;
\sum_{p,q=1}^n f^{pq} f_{pqj}.
\]
All such summations are taken over tangent indices unless otherwise specified.

Formula~\eqref{eq:AN-analytic-sec2} makes the computational structure explicit.
The first $n$ components are obtained by solving a tangent linear system,
while the right-hand side contains the third-order contraction
\[
f^{pq}(z)f_{pqj}(z),
\]
which couples the inverse tangent Hessian with third-order derivatives.
This term is the main obstacle to scalable affine normal computation.

\subsection{A log-determinant identity for the third-order contraction}
\label{sec:logdet}

We now show that the third-order contraction in
\eqref{eq:AN-analytic-sec2} is in fact a second-order quantity.
This is the key reduction underlying the computational framework of this paper.

\begin{lemma}[Log-determinant gradient identity]
\label{lem:logdet-sec2}
Assume that the tangent Hessian block $H_T(x)$ is invertible
in a neighborhood of $z$.
Then for each tangent index $i\in\{1,\dots,n\}$,
\begin{equation}
\label{eq:logdet-identity-sec2}
\partial_{i}\log\det(H_T)(z)
=
\sum_{p,q=1}^n f^{pq}(z)\,f_{pqi}(z) = \mathrm{tr}\!\left(H_T(z)^{-1}\,\partial_{i}H_T(z)\right).
\end{equation}
\end{lemma}

\begin{proof}
Consider the matrix-valued function
\[
A(x):=H_T(x)=[f_{pq}(x)]_{p,q=1}^n.
\]
Fix a tangent index $i\in\{1,\dots,n\}$ and consider the curve
\[
x(t)=z+t e_i,
\]
where $e_i$ is the $i$-th coordinate vector.
Define the matrix-valued function of a scalar variable
\[
A(t):=A(x(t))=H_T(z+t e_i).
\]
Then, by the chain rule,
\[
A'(0) =
\frac{d}{dt}A(t)\Big|_{t=0}
=
\partial_{i} H_T(z).
\]
Moreover, for each entry we have
\[
[A'(0)]_{pq}
=
\frac{d}{dt}f_{pq}(z+t e_i)\Big|_{t=0}
=
\partial_{i} f_{pq}(z)
=
f_{pqi}(z).
\]
Since $A(0)=H_T(z)$ is invertible, Jacobi's formula yields
\[
\frac{d}{dt}\log\det A(t)\Big|_{t=0}
=
\mathrm{tr}\!\left(A(0)^{-1}A'(0)\right).
\]
Substituting $A(0)=H_T(z)$ and $A'(0)=\partial_{i}H_T(z)$ gives
\[
\partial_{i}\log\det(H_T)(z)
=
\mathrm{tr}\!\left(H_T(z)^{-1}\,\partial_{i}H_T(z)\right)
=
\sum_{p,q=1}^n f^{pq}(z)\,f_{pqi}(z),
\]
which is exactly \eqref{eq:logdet-identity-sec2}.
\end{proof}

\begin{remark}
Lemma~\ref{lem:logdet-sec2} is a direct consequence of the classical
Jacobi formula from matrix calculus.
Its significance here is not the identity itself,
but the fact that it removes the need to explicitly construct
the third-order tensor contraction appearing in
the affine normal formula.
In particular, it shows that the affine normal depends only on
the gradient of the local log-volume distortion
\(
\log\det(H_T)
\),
and hence is fundamentally governed by second-order affine curvature information.
\end{remark}

\subsection{Equivalent linear-system form}

Using Lemma~\ref{lem:logdet-sec2},
we introduce the vectors
\[
h\in\mathbb{R}^n,
\qquad
h_i := f_{n+1,i}(z),
\]
and
\[
a\in\mathbb{R}^n,
\qquad
a_i := \partial_{i}\log\det(H_T)(z).
\]
Then \eqref{eq:AN-analytic-sec2} can be rewritten as
\begin{equation}
\label{eq:AN-linear-system-sec2}
d_{\mathrm{AN}}(z)\ \propto\
\begin{pmatrix}
H_T(z)^{-1}\Big(h-\frac{\|\nabla f(z)\|}{n+2}\,a\Big)
\\[4pt]
-1
\end{pmatrix}.
\end{equation}
Equivalently, if we define
\[
u:=H_T(z)^{-1}\Big(h-\frac{\|\nabla f(z)\|}{n+2}\,a\Big),
\]
then $u$ is characterized by the tangent linear system
\begin{equation}
\label{eq:tangent-linear-system-sec2}
H_T(z)\,u
=
h-\frac{\|\nabla f(z)\|}{n+2}\,a.
\end{equation}
Thus, the analytic affine normal is obtained by solving
a tangent linear system with a right-hand side
built from a mixed second-order term $h$
and a log-determinant curvature term $a$.

\subsection{Coordinate-free and implementation-friendly form}

For large-scale computation it is undesirable to explicitly rotate coordinates
into a normal-aligned frame at every iterate.
We therefore reformulate the affine normal direction in a coordinate-free way. Let
\[
g:=\nabla f(z)\in\mathbb{R}^{n+1},\qquad
\nu:=\frac{g}{\|g\|}\in\mathbb{R}^{n+1}
\]
be the Euclidean gradient and the unit normal direction.
Choose any orthogonal matrix
\[
Q=[T\ \ \nu]\in\mathbb{R}^{(n+1)\times(n+1)},
\]
where
\[
T\in\mathbb{R}^{(n+1)\times n}
\]
has orthonormal columns spanning the tangent space $T_zM_z$.
Let
\[
H:=\nabla^2 f(z).
\]
Then the tangent Hessian block and the mixed second-order term can be written as
\begin{equation}
\label{eq:HT-h-coordinate-free}
H_T = T^\top H T,
\qquad
h = T^\top H \nu.
\end{equation}
The log-determinant gradient $a$ is understood as the gradient of
\(\log\det(H_T)\) in tangent coordinates. If we define
\begin{equation}
\label{eq:u-coordinate-free}
u:=H_T^{-1}\Big(h-\frac{\|g\|}{n+2}\,a\Big)\in\mathbb{R}^n,
\end{equation}
then the affine normal direction in ambient coordinates is given by
\begin{equation}
\label{eq:AN-ambient-sec2}
d_{\mathrm{AN}}^{\mathrm{amb}}(z)\ \propto\ Tu-\nu.
\end{equation}
The next proposition shows that this coordinate-free expression
is exactly equivalent to the normal-aligned formula
\eqref{eq:AN-analytic-sec2}.

\begin{proposition}[Equivalence of the aligned and ambient representations]
\label{prop:AN-equiv-sec2}
Let $f:\mathbb{R}^{n+1}\to\mathbb{R}$ be $C^3$ and let $z$ satisfy
\(\nabla f(z)\neq 0\).
Then the ambient direction \eqref{eq:AN-ambient-sec2},
with $u$ defined by \eqref{eq:u-coordinate-free},
coincides with the affine normal direction
given by \eqref{eq:AN-analytic-sec2}
up to a nonzero scalar factor.
\end{proposition}

\begin{proof}
Since $Q=[T\ \ \nu]$ is orthogonal, the coordinates
\[
y=Q^\top(x-z)
\]
define a normal-aligned coordinate system. In these coordinates, the Hessian transforms as
\[
H_y = Q^\top H Q
=
\begin{pmatrix}
T^\top H T & T^\top H \nu\\
\nu^\top H T & \nu^\top H \nu
\end{pmatrix}
=
\begin{pmatrix}
H_T & h\\
h^\top & \nu^\top H \nu
\end{pmatrix}.
\]
Hence the tangent block in the aligned formula is precisely $H_T$,
and the mixed term is exactly $h$.
By Lemma~\ref{lem:logdet-sec2}, the third-order contraction term in
\eqref{eq:AN-analytic-sec2} is the tangent log-determinant gradient $a$.
Therefore the first $n$ components of the affine normal direction
in aligned coordinates are precisely the vector $u$
defined by \eqref{eq:u-coordinate-free}. Thus, in aligned coordinates, the affine normal direction is
\[
\binom{u}{-1}.
\]
Transforming this vector back to ambient coordinates gives
\[
Q\binom{u}{-1}
=
Tu-\nu,
\]
which is exactly \eqref{eq:AN-ambient-sec2}.
\end{proof}

\subsection{Matrix-free viewpoint}

The coordinate-free formula immediately yields an implementation-friendly,
matrix-free perspective.
Indeed, affine normal computation does not require explicitly forming
the inverse of $H_T$.
It suffices to provide the tangent operator
\begin{equation}
\label{eq:matrix-free-operator-sec2}
\mathcal{H}_T(v):=T^\top\big(H(Tv)\big),
\qquad v\in\mathbb{R}^n,
\end{equation}
and solve the linear system
\[
\mathcal{H}_T(u)=b,
\qquad
b:=h-\frac{\|g\|}{n+2}\,a.
\]

This reformulation is crucial for high-dimensional settings.
Once a routine for Hessian--vector multiplication
\[
y\mapsto H\,y
\]
is available, the entire affine normal direction can be computed
without explicitly constructing either $H_T$ or $H_T^{-1}$.
This observation will serve as the basis for the stochastic
and polynomially structured algorithms developed in the following sections.

\section{Polynomial Structure and Closed-Form Derivative Kernels}
\label{sec:poly}

In this section we exploit the algebraic structure of polynomial objectives
to derive explicit formulas for derivatives of all orders
and to develop scalable matrix-free computational kernels.
These results form the computational backbone of our affine normal framework.

\subsection{Polynomial model and notation}

Let $d:=n+1$.
We consider polynomial objectives of the form
\begin{equation}
\label{eq:poly-model}
f(x)=\sum_{\ell=1}^m c_\ell\,x^{\alpha^{(\ell)}},
\qquad
x^{\alpha}:=\prod_{i=1}^d x_i^{\alpha_i},
\end{equation}
where $m$ is the number of monomials, $c_\ell\in\mathbb{R}$, and
$\alpha^{(\ell)}\in\mathbb{N}^d$ are multi-indices.
We denote the support of each monomial by
\[
\supp(\alpha):=\{i:\alpha_i>0\}.
\]
We further define the average support size by
\[
s:=\frac{1}{m}\sum_{\ell=1}^m |\supp(\alpha^{(\ell)})|.
\]
To describe derivatives compactly,
we introduce the falling factorial notation
\[
(r)_k := r(r-1)\cdots(r-k+1),\qquad (r)_0=1,
\]
and for multi-indices $\beta\in\mathbb{N}^d$, we define
\[
(\alpha)_\beta := \prod_{i=1}^d (\alpha_i)_{\beta_i},\qquad
\partial^\beta
:=
\frac{\partial^{|\beta|}}{\partial x_1^{\beta_1}\cdots \partial x_d^{\beta_d}},
\qquad
|\beta| := \sum_{i=1}^d \beta_i.
\]

Throughout this section, all derivatives are taken with respect to the ambient coordinates $x \in \mathbb{R}^{n+1}$. In particular, $\partial_i$ and $\partial^\beta$ denote standard partial derivatives in the ambient space. The tangent quantities appearing in the affine normal construction are subsequently obtained via projection or restriction of these ambient derivatives onto the tangent space.

\subsection{Closure of polynomial derivatives}

A fundamental property of polynomial functions
is that derivatives of any order remain within the same monomial family.

\begin{lemma}[Algebraic closure of polynomial derivatives]
\label{lem:poly-closure}
Let $\alpha,\beta\in\mathbb{N}^d$.
Then
\[
\partial^\beta x^\alpha
=
\begin{cases}
(\alpha)_\beta\,x^{\alpha-\beta}, & \beta\le\alpha,\\
0, & \text{otherwise}.
\end{cases}
\]
Consequently, for any polynomial $f$ of the form \eqref{eq:poly-model},
all derivatives $\partial^\beta f$ remain sparse linear combinations
of monomials.
\end{lemma}

\begin{proof}
For a multi-index $\beta\in\mathbb{N}^d$, we have
\[
\partial^\beta x^\alpha
=
\prod_{i=1}^d \frac{\partial^{\beta_i}}{\partial x_i^{\beta_i}} x_i^{\alpha_i}.
\]
By the one-dimensional identity
\[
\frac{d^k}{dx^k}x^r = (r)_k x^{r-k},
\]
each factor evaluates to
\[
\frac{\partial^{\beta_i}}{\partial x_i^{\beta_i}} x_i^{\alpha_i}
=
\begin{cases}
(\alpha_i)_{\beta_i}\,x_i^{\alpha_i-\beta_i}, & \beta_i\le \alpha_i,\\
0, & \text{otherwise}.
\end{cases}
\]
Taking the product over $i=1,\dots,d$ yields the result.
\end{proof}

\begin{remark}
Lemma~\ref{lem:poly-closure} shows that polynomial derivatives are closed under differentiation and preserve sparsity at the level of monomial supports.
As a consequence, all derivative operators involved in affine normal computation
can be implemented via sparse monomial contractions,
without introducing new functional structures.
This property is the key reason why polynomial objectives
admit scalable implementations.
\end{remark}

\subsection{Gradient and Hessian formulas}

Let $t_\ell(x):=c_\ell x^{\alpha^{(\ell)}}$.
Using Lemma~\ref{lem:poly-closure}, we obtain explicit expressions.

\paragraph{Gradient.}
For each coordinate $i$,
\[
\partial_{i} t_\ell(x)
=
\begin{cases}
c_\ell\,\alpha_{\ell i}\,x^{\alpha^{(\ell)}-e_i}, & \alpha_{\ell i}>0,\\
0, & \text{otherwise}.
\end{cases}
\]

\paragraph{Hessian.}
For $i,j\in\{1,\dots,d\}$,
\[
\partial_{i}\partial_{j} t_\ell(x)
=
\begin{cases}
c_\ell\,\alpha_{\ell i}(\alpha_{\ell i}-1)\,x^{\alpha^{(\ell)}-2e_i}, & i=j,\\
c_\ell\,\alpha_{\ell i}\alpha_{\ell j}\,x^{\alpha^{(\ell)}-e_i-e_j}, & i\neq j,
\end{cases}
\]
with the convention that terms with negative exponents are interpreted as zero.

\subsection{Efficient Hessian--vector product}

The key computational primitive in our framework
is the Hessian--vector product
\[
w = \nabla^2 f(x)\,v.
\]
We now derive a form suitable for sparse implementation.

For a single monomial $t(x)=c x^\alpha$,
we first observe that its gradient admits the representation
\[
\nabla t(x)
=
t(x)\,\Big(\frac{\alpha_1}{x_1},\dots,\frac{\alpha_d}{x_d}\Big).
\]
Thus, for any direction $v\in\mathbb{R}^d$,
\[
\nabla t(x)^\top v
=
t(x)\sum_{j\in\supp(\alpha)} \alpha_j \frac{v_j}{x_j}.
\]
This motivates the definition
\[
\beta(x;v):=\sum_{j\in\supp(\alpha)} \alpha_j\,\frac{v_j}{x_j},
\]
so that
\[
\nabla t(x)^\top v = t(x)\,\beta(x;v).
\]
Differentiating once more yields
\[
\nabla^2 t(x)\,v
=
\nabla\big(t(x)\,\beta(x;v)\big).
\]
Applying the product rule, we obtain
\[
\nabla^2 t(x)\,v
=
(\nabla t(x))\,\beta(x;v)
+
t(x)\,\nabla \beta(x;v).
\]
A direct computation shows that
\[
\partial_{i} \beta(x;v)
=
-\alpha_i\,\frac{v_i}{x_i^2},
\qquad i\in\supp(\alpha).
\]
Substituting the above expressions yields the component-wise formula
\begin{equation}
\label{eq:poly-Hv}
(\nabla^2 t(x)\,v)_i
=
t(x)\Big(\frac{\alpha_i}{x_i}\Big)\beta(x;v)
-
t(x)\Big(\frac{\alpha_i}{x_i^2}\Big)v_i,
\qquad i\in\supp(\alpha).
\end{equation}

\begin{proposition}[Matrix-free Hessian--vector kernel]
\label{prop:Hv-kernel}
For a polynomial of the form \eqref{eq:poly-model},
the Hessian--vector product $w=\nabla^2 f(x)\,v$ can be computed in
\[
O\!\Big(\sum_{\ell=1}^m |\supp(\alpha^{(\ell)})|\Big)
\]
operations using the update formula \eqref{eq:poly-Hv}.
\end{proposition}

\begin{proof}
By linearity, we have
\[
\nabla^2 f(x)\,v
=
\sum_{\ell=1}^m \nabla^2 t_\ell(x)\,v,
\quad
t_\ell(x)=c_\ell x^{\alpha^{(\ell)}}.
\]
For each monomial $t_\ell$, the quantities $t_\ell(x)$ and $\beta(x;v)$
can be evaluated in $O(|\supp(\alpha^{(\ell)})|)$ operations.
The update \eqref{eq:poly-Hv} affects only indices
$i\in\supp(\alpha^{(\ell)})$, and each update requires $O(1)$ work.
Therefore, the total cost per monomial is $O(|\supp(\alpha^{(\ell)})|)$.
Summing over $\ell=1,\dots,m$ yields the result.
\end{proof}

\begin{remark}
The computational cost depends only on the total sparsity
\[
\sum_{\ell=1}^m |\supp(\alpha^{(\ell)})|,
\]
and is independent of the ambient dimension $d$.
In contrast, explicit dense Hessian methods require $O(d^2)$ operations.
\end{remark}

\subsection{Third-order contraction via directional derivatives}
\label{sec:third-order-kernel}

Affine normal computation requires evaluating expressions of the form
\[
D^3 f(x)[u,v,\cdot],
\]
where $D^3 f(x)$ denotes the third derivative of $f$ viewed as a symmetric trilinear form, and
$D^3 f(x)[u,v,\cdot]\in\mathbb{R}^d$ denotes the vector obtained by contracting the first two slots with $u$ and $v$.
Equivalently, its $k$-th component is
\[
\bigl(D^3 f(x)[u,v,\cdot]\bigr)_k
=
\sum_{i,j=1}^d \partial_{ijk}f(x)\,u_i v_j.
\]
Here $\partial_{ijk}f(x)$ denotes the ambient third-order partial derivative in the standard Euclidean coordinates of $\mathbb{R}^d$, rather than the tangent-coordinate notation used in Section~\ref{sec:analytic}.
Such quantities arise naturally in the stochastic trace formulation
introduced in Section~\ref{sec:stochastic}.
Although the analytic affine-normal formula in Section~\ref{sec:analytic}
involves only tangent-index third-order terms such as $f_{pqj}$,
it is computationally more convenient to work first with the ambient vector
$D^3 f(x)[u,v,\cdot]$ and then project it onto the tangent space.
Indeed, the quantity used in the affine-normal computation is
\[
T^\top D^3 f(x)[Tu,Tv,\cdot],
\]
which is exactly the tangent component corresponding to the
third-order contractions appearing in the normal-aligned formula.
This ambient-to-tangent formulation is what makes the subsequent
matrix-free polynomial kernels straightforward to derive and implement.

Instead of explicitly forming third-order tensors,
we use the identity
\[
D^3 f(x)[u,v,\cdot]
=
\nabla\big(u^\top (\nabla^2 f(x)) v\big).
\]

To obtain an explicit sparse kernel, we first pass from the full polynomial
$f$ to a single monomial. This is sufficient because
$D^3 f(x)[u,v,\cdot]$ depends linearly on $f$, so the polynomial case is
recovered by summing the monomial contributions.
We therefore first derive a matrix-free representation at the level of a single monomial.
Let $t(x)=c x^\alpha$ and assume $x_i\neq 0$ for all $i\in\supp(\alpha)$.
Define
\[
A_u := \sum_{i\in\supp(\alpha)} \alpha_i\,\frac{u_i}{x_i},
\qquad
A_v := \sum_{i\in\supp(\alpha)} \alpha_i\,\frac{v_i}{x_i},\qquad
C := \sum_{i\in\supp(\alpha)} \alpha_i\,\frac{u_i v_i}{x_i^2}.
\]
Using the Hessian--vector representation from the previous subsection, we first express the scalar contraction
\[
u^\top (\nabla^2 t(x)) v
=
t(x)\big(A_u A_v - C\big).
\]
We then differentiate this scalar quantity with respect to $x_k$ in order to recover the
vector $D^3 t(x)[u,v,\cdot]$. Thus
\begin{equation}
\label{eq:d3-kernel}
(\nabla^3 t(x)[u,v,\cdot])_k
=
t(x)\,\alpha_k\Big(
\frac{A_u A_v - C}{x_k}
-\frac{u_k A_v + v_k A_u}{x_k^2}
+\frac{2u_k v_k}{x_k^3}
\Big),
\qquad k\in\supp(\alpha).
\end{equation}

\begin{proposition}[Matrix-free evaluation of third-order contractions]
\label{prop:d3-kernel}
For a polynomial of the form \eqref{eq:poly-model},
the vector $D^3 f(x)[u,v,\cdot]$ can be evaluated in
\[
O\!\Big(\sum_{\ell=1}^m |\supp(\alpha^{(\ell)})|\Big)
= O(ms)
\]
operations using the representation \eqref{eq:d3-kernel}.
\end{proposition}

\begin{proof}
By linearity, it suffices to estimate the cost of evaluating
the contribution of a single monomial.
Consider
\[
t_\ell(x)=c_\ell x^{\alpha^{(\ell)}}, \qquad S_\ell:=\supp(\alpha^{(\ell)}).
\]
The quantities
\[
A_u^{(\ell)}=\sum_{i\in S_\ell}\alpha_i^{(\ell)}\frac{u_i}{x_i},
\qquad
A_v^{(\ell)}=\sum_{i\in S_\ell}\alpha_i^{(\ell)}\frac{v_i}{x_i},
\qquad
C^{(\ell)}=\sum_{i\in S_\ell}\alpha_i^{(\ell)}\frac{u_i v_i}{x_i^2},
\]
as well as the monomial value $t_\ell(x)$,
can each be evaluated in $O(|S_\ell|)$ operations.
The update \eqref{eq:d3-kernel} affects only indices $k\in S_\ell$,
and each update requires $O(1)$ work.
Hence the total cost per monomial is $O(|S_\ell|)$.
Summing over $\ell=1,\dots,m$ yields
\[
O\!\Big(\sum_{\ell=1}^m |S_\ell|\Big)
=
O\!\Big(\sum_{\ell=1}^m |\supp(\alpha^{(\ell)})|\Big).
\]
Recalling that
\[
s:=\frac{1}{m}\sum_{\ell=1}^m |\supp(\alpha^{(\ell)})|,
\]
we have
\[
\sum_{\ell=1}^m |\supp(\alpha^{(\ell)})| = ms.
\]
Therefore the total cost is
\[
O(ms).
\]
\end{proof}

\begin{remark}
This result shows that the directional third-order information required
by affine normal computation can be evaluated at essentially the same
computational cost as Hessian--vector products, up to constant factors.
This constitutes a significant reduction: while naive third-order tensor
operations scale as $O(d^3)$, the structured polynomial representation
admits a fully matrix-free implementation with linear complexity in the sparsity.
Such a reduction critically relies on the multiplicative structure of monomials
and does not extend to general smooth functions.
\end{remark}

The analysis above shows that the cost of estimating
the log-determinant gradient is dominated by repeated
Hessian--vector products and Krylov solves.
For polynomial objectives, this yields a matrix-free implementation
whose full end-to-end complexity is established in Section~\ref{sec:complexity}.

\section{Stochastic Trace Estimation and Krylov Solvers}
\label{sec:stochastic}

In this section we develop a scalable matrix-free strategy
for computing the log-determinant gradient term
\[
a=\nabla_t \log\det(H_T),
\]
which is the only remaining nontrivial component
in the affine normal direction. For related large-scale approaches to trace and log-determinant estimation, see, for example,
\cite{bai1996some,han2015large,ubaru2017fast}.

\subsection{Trace formulation of the log-determinant gradient}

From Lemma~\ref{lem:logdet-sec2},
we have
\[
a_i = \partial_{i} \log\det(H_T)
= \mathrm{tr}\!\big(H_T^{-1}\,\partial_{i}H_T\big).
\]
Define
\[
M_i := H_T^{-1}\,\partial_{i}H_T.
\]
Then
\begin{equation}
\label{eq:trace-form}
a_i = \mathrm{tr}(M_i).
\end{equation}

This representation eliminates explicit third-order tensor construction
and reduces the problem to trace estimation of matrix products.

\subsection{Hutchinson stochastic trace estimator}

To estimate \eqref{eq:trace-form}, we use the Hutchinson trace estimator \cite{girard1989trace,hutchinson1990trace}.
Let $\xi\in\mathbb{R}^n$ be a random vector with independent
Rademacher entries ($\pm1$ with equal probability).
Then for any matrix $M$,
\[
\mathbb{E}[\xi^\top M \xi] = \mathrm{tr}(M).
\]
Using $q$ independent samples $\{\xi^{(\ell)}\}_{\ell=1}^q$,
we approximate
\begin{equation}
\label{eq:hutch}
a_i \approx \frac{1}{q}\sum_{\ell=1}^q
(\xi^{(\ell)})^\top M_i \xi^{(\ell)}.
\end{equation}

The estimator is unbiased, and its variance is governed by the off-diagonal structure of $M_i$; see, for instance, \cite{avron2011randomized}.
In practice, small-to-moderate values of $q$ (e.g., $5$--$20$) often provide a useful accuracy--cost tradeoff when only sufficiently accurate search directions are needed.

\subsection{Reduction to linear systems and third-order contractions}

Substituting $M_i=H_T^{-1}\partial_{i}H_T$ into \eqref{eq:hutch},
we obtain
\[
(\xi^\top M_i \xi)
=
\xi^\top H_T^{-1}(\partial_{i}H_T)\xi.
\]

Let
\[
y := H_T^{-1}\xi,
\]
which is obtained by solving the linear system
\[
H_T y = \xi.
\]
Then
\begin{equation}
\label{eq:bilinear}
\xi^\top M_i \xi
=
y^\top (\partial_{i}H_T)\xi.
\end{equation}
Thus each stochastic probe requires:
\begin{itemize}
\item one linear system solve,
\item one third-order directional contraction.
\end{itemize}
Using the identity
\[
D^3 f(x)[u,v,\cdot]
=
\nabla\big(u^\top (\nabla^2 f(x)) v\big),
\]
the action of $\partial_i H_T$ on $\xi$ is recovered from the
ambient third-order directional derivative together with tangent projection.
More precisely, for tangent vectors represented in ambient coordinates,
the vector whose $i$-th component is $y^\top (\partial_i H_T)\xi$
corresponds to
\[
T^\top D^3 f(x)[Ty,T\xi,\cdot],
\]
which is exactly the type of third-order contraction treated in
Section~\ref{sec:poly}. Hence the required quantity can be evaluated
via the matrix-free polynomial kernels developed there.

\subsection{Krylov methods for linear system solves}

The system
\[
H_T y = \xi
\]
is not formed explicitly.
Instead, it is solved by Krylov subspace methods \cite{hestenes1952cg,paige1975minres,saad1986gmres,saad2003}.

\paragraph{Symmetric positive definite case.}
When $H_T\succ 0$ (the elliptic region), we use the conjugate gradient (CG) method \cite{hestenes1952cg}.
Let $y_k$ denote the $k$-th CG iterate.
Starting from an initial guess $y_0$, the residual and search direction are initialized by
\[
r_0=\xi-H_T y_0,
\qquad
p_0=r_0,
\]
and the iteration takes the form
\[
\alpha_k=\frac{r_k^\top r_k}{p_k^\top H_T p_k},
\qquad
y_{k+1}=y_k+\alpha_k p_k,
\qquad
r_{k+1}=r_k-\alpha_k H_T p_k,
\]
\[
\beta_{k+1}=\frac{r_{k+1}^\top r_{k+1}}{r_k^\top r_k},
\qquad
p_{k+1}=r_{k+1}+\beta_{k+1}p_k.
\]
Then the error satisfies
\begin{equation}
\label{eq:cg-rate}
\|y-y_k\|_{H_T}
\le
2\Big(\frac{\sqrt{\kappa}-1}{\sqrt{\kappa}+1}\Big)^k
\|y\|_{H_T},
\end{equation}
where $\kappa$ is the condition number of $H_T$.

\paragraph{Indefinite case.}
When $H_T$ is indefinite, we use MINRES \cite{paige1975minres} or GMRES \cite{saad1986gmres}, or introduce a regularization
\[
H_T+\lambda I
\]
to ensure numerical stability.
These methods generate iterates in the Krylov subspace
\[
\mathcal{K}_k(H_T,r_0)
:=
\mathrm{span}\{r_0,H_T r_0,\dots,H_T^{k-1}r_0\},
\]
with $r_0=\xi-H_T y_0$, and differ mainly in the residual minimization principle used within this subspace.

\begin{remark}
The matrix-free formulation requires only Hessian--vector products, which can be implemented efficiently for polynomial objectives.
At the Krylov level, no explicit assembly of the ambient Hessian or of any third-order tensor is required.
\end{remark}

\subsection{Error decomposition}

The total error in estimating $a_i$ arises from two sources:

\paragraph{(i) stochastic trace error.}
\[
\mathrm{Var}\big(\widehat{a}_i\big)
=
O\!\left(\frac{\|M_i\|_F^2}{q}\right).
\]

\paragraph{(ii) linear solver error.}
Let $y_k$ be the approximate solution. In the symmetric positive definite
or regularized case, combining \eqref{eq:bilinear} with the duality between
the $H_T$- and $H_T^{-1}$-norms gives
\[
|y^\top(\partial_{i}H_T)\xi
-
y_k^\top(\partial_{i}H_T)\xi|
\le
\|y-y_k\|_{H_T}\,\|(\partial_{i}H_T)\xi\|_{H_T^{-1}}.
\]
Together with \eqref{eq:cg-rate}, this shows that the linear-solver contribution
decays at the Krylov rate. Hence the total error can be controlled by increasing
the number of probes $q$ and tightening the Krylov solver tolerance.

\subsection{Overall computational complexity}

Combining Sections~\ref{sec:analytic}--\ref{sec:poly}, each stochastic probe requires
\begin{itemize}
\item $k$ Hessian--vector products (corresponding to Krylov iterations),
\item one third-order directional kernel evaluation.
\end{itemize}
Therefore, the total cost of estimating $a$ is
\[
O\bigl(qk\cdot \mathrm{cost}(Hv)\bigr).
\]
For polynomial objectives,
\[
\mathrm{cost}(Hv)=O(m s),
\]
which yields
\[
\mathrm{cost}(a)=O(qk\,m s).
\]

The discussion above identifies the cost of the log-determinant gradient subroutine:
for polynomial objectives, its dominant complexity is proportional to $qk\,m s$.
This subroutine is the key scalable ingredient in MF-LogDet; the full framework is presented in Section~\ref{sec:algorithm}.

A full end-to-end complexity analysis of the complete affine normal computation,
including the final tangent solve and structural scaling consequences,
is deferred to Section~\ref{sec:complexity}.

\section{Algorithmic Framework and Implementation}
\label{sec:algorithm}

In this section we present a complete matrix-free algorithm
for computing affine normal directions,
together with practical implementation strategies.
The formulation integrates the analytic reduction of Section~\ref{sec:analytic},
the polynomial kernels of Section~\ref{sec:poly},
and the stochastic trace estimation of Section~\ref{sec:stochastic}.

\subsection{Overview of the computational pipeline}

At a point $x\in\mathbb{R}^d$, the affine normal direction is given by
\[
d_{\mathrm{AN}}^{\mathrm{amb}}
\propto
Tu - \nu,
\]
where
\[
u = H_T^{-1}\Big(h-\frac{\|\nabla f(x)\|}{n+2}a\Big).
\]

Thus the computation consists of three main components:
\begin{itemize}
\item evaluation of $h = T^\top (\nabla^2 f(x)) \nu$,
\item estimation of $a = \nabla_t \log\det(H_T)$,
\item solution of the linear system $H_T u = b$.
\end{itemize}

Each of these steps admits a matrix-free implementation.

\subsection{Matrix-free operators}

Let $H=\nabla^2 f(x)$.
We define the tangent operator
\begin{equation}
\label{eq:HT-operator}
\mathcal{H}_T(v)
:=
T^\top\big(H(Tv)\big),
\qquad v\in\mathbb{R}^{d-1}.
\end{equation}

This operator can be applied using:
\begin{itemize}
\item one lifting $v \mapsto Tv$,
\item one Hessian--vector product,
\item one projection $T^\top(\cdot)$.
\end{itemize}

Thus no explicit formation of $H_T$ is required.

\subsection{Matrix-free log-determinant affine normal computation}

We now present the matrix-free log-determinant (MF-LogDet) procedure for affine normal computation.
The method consists of a common outer structure together with two possible ways to evaluate the tangent log-determinant gradient
\[
a=\nabla_t \log\det(H_T),
\]
namely, an \emph{exact tangent-space evaluation} and a \emph{stochastic Hutchinson approximation}. Recall that the ambient dimension is $d=n+1$, so that the coefficient $(n+2)^{-1}$ in the classical formula becomes $(d+1)^{-1}$ in the present notation.

\begin{algorithm}[h!]
\caption{MF-LogDet affine normal computation}
\label{alg:mf-logdet}
\KwIn{Point $x$, polynomial representation of $f$, routine for evaluating $a=\nabla_t \log\det(H_T)$ (exact or stochastic), regularization parameter $\lambda$}
\KwOut{Affine normal direction $d_{\mathrm{AN}}(x)$}

\BlankLine
\textbf{Step 1: Gradient and normal direction}\;
Compute
\[
g=\nabla f(x),
\qquad
\nu=\frac{g}{\|g\|}.
\]

\BlankLine
\textbf{Step 2: Tangent basis}\;
Construct an orthonormal basis
\[
T\in\mathbb{R}^{d\times(d-1)}
\]
such that
\[
T^\top \nu=0.
\]

\BlankLine
\textbf{Step 3: Tangent Hessian quantities}\;
Form the tangent Hessian operator
\[
H_T = T^\top \nabla^2 f(x)\, T
\]
and compute
\[
h = T^\top (\nabla^2 f(x)\, \nu).
\]

\BlankLine
\textbf{Step 4: Tangent log-determinant gradient}\;
Compute
\[
a=\nabla_t \log\det(H_T),
\]
either by the exact tangent-space formula (Algorithm~\ref{alg:mf-logdet-exact})
or by the stochastic Hutchinson approximation (Algorithm~\ref{alg:mf-logdet-hutch}).

\BlankLine
\textbf{Step 5: Tangent linear system}\;
Set
\[
b = h - \frac{\|g\|}{d+1} a,
\]
and solve
\[
(H_T+\lambda I)u=b.
\]

\BlankLine
\textbf{Step 6: Affine normal assembly}\;
Set
\[
d_{\mathrm{AN}}(x)=Tu-\nu.
\]

\Return $d_{\mathrm{AN}}(x)$\;
\end{algorithm}

\begin{algorithm}[h!]
\caption{Exact evaluation of $a=\nabla_t \log\det(H_T)$}
\label{alg:mf-logdet-exact}
\KwIn{Point $x$, tangent basis $T$, tangent Hessian block $H_T$, regularization parameter $\lambda$}
\KwOut{Exact tangent log-determinant gradient $a$}

\BlankLine
Initialize
\[
a\leftarrow 0.
\]

\For{$j=1,\dots,d-1$}{
    Let $e_j$ denote the $j$-th canonical basis vector in $\mathbb{R}^{d-1}$\;

    Solve
    \[
    (H_T+\lambda I)y^{(j)}=e_j.
    \]

    Lift
    \[
    u^{(j)}=Ty^{(j)},
    \qquad
    v^{(j)}=Te_j.
    \]

    Compute the directional third-derivative vector
    \[
    w^{(j)} = D^3f(x)[u^{(j)},v^{(j)},\cdot].
    \]

    Update
    \[
    a \leftarrow a + T^\top w^{(j)}.
    \]
}
\Return $a$\;
\end{algorithm}

\begin{algorithm}[h!]
\caption{Stochastic evaluation of $a=\nabla_t \log\det(H_T)$}
\label{alg:mf-logdet-hutch}
\KwIn{Point $x$, tangent basis $T$, tangent Hessian operator $H_T$, probe count $q$, regularization parameter $\lambda$}
\KwOut{Stochastic approximation of the tangent log-determinant gradient $a$}

\BlankLine
Initialize
\[
a\leftarrow 0.
\]

\For{$\ell=1,\dots,q$}{
    Sample a Rademacher probe
    \[
    \xi^{(\ell)}\in\mathbb{R}^{d-1}.
    \]

    Solve
    \[
    (H_T+\lambda I)y^{(\ell)}=\xi^{(\ell)}
    \]
    by CG/PCG.

    Lift
    \[
    u^{(\ell)}=Ty^{(\ell)},
    \qquad
    v^{(\ell)}=T\xi^{(\ell)}.
    \]

    Compute the directional third-derivative vector
    \[
    w^{(\ell)} = D^3f(x)[u^{(\ell)},v^{(\ell)},\cdot].
    \]

    Update
    \[
    a \leftarrow a + T^\top w^{(\ell)}.
    \]
}
\Return
\[
a/q.
\]
\end{algorithm}

\subsection{Computational structure}

The MF-LogDet formulation reduces affine normal computation to a small set of basic linear-algebraic primitives.
At the implementation level, the dominant operations are:
\begin{itemize}
\item gradient evaluation,
\item Hessian--vector products,
\item evaluation of the tangent log-determinant gradient $a$, either exactly or stochastically,
\item inner products and tangent-space projections,
\item and linear solves involving the tangent Hessian operator.
\end{itemize}

The key point is that the geometric quantity
\[
a=\nabla_t \log\det(H_T)
\]
is never computed through explicit third-order tensor contraction.
Instead, it is recovered either from an exact tangent-space trace formula or from a stochastic trace estimator. In these procedures, directional third-order contractions may appear as internal matrix-free primitives. Their efficient evaluation relies on the techniques of Section~\ref{sec:poly}, especially the matrix-free kernel of Section~\ref{sec:third-order-kernel}, which yields $O(ms)$ complexity in the sparse polynomial regime. Nevertheless, the algorithm never forms or directly computes a full third-order tensor as a standalone object.

\begin{remark}
The exact and stochastic variants of MF-LogDet differ in how the tangent log-determinant gradient is computed.

In the exact variant, one explicitly forms the tangent Hessian block
\[
H_T = T^\top \nabla^2 f(x) T
\]
and evaluates the tangent log-determinant gradient by the exact basis-wise trace formula, implemented through solves against canonical basis vectors and the corresponding directional third-order contractions.
This version is useful for formula-level verification and medium-scale benchmarking.

In the stochastic variant, one avoids explicit trace computation and instead estimates
\[
\nabla_t \log\det(H_T)
\]
through repeated linear solves and directional third-derivative evaluations.
This is the version relevant for scalability.
In particular, it avoids forming any full third-order tensor and reduces the dominant cost to repeated operator applications.
\end{remark}

\subsection{Parameter selection}

The stochastic MF-LogDet implementation involves three main numerical parameters: the number of Hutchinson probes $q$, the Krylov iteration budget (or tolerance), and the regularization parameter $\lambda$.

\paragraph{Number of probes $q$.}
The probe count $q$ controls the accuracy of the stochastic trace approximation.
Larger values of $q$ reduce the variance of the estimated log-determinant gradient, but increase the computational cost almost proportionally.
Our numerical experiments indicate that small-to-moderate values of $q$ already provide a useful accuracy--cost tradeoff, while very large probe counts yield diminishing returns.

\paragraph{Krylov iterations.}
Each stochastic probe requires the solution of a linear system involving the regularized tangent Hessian.
Hence the overall cost depends on the number of Krylov iterations needed for these inner solves.
This number is governed primarily by the conditioning of the regularized tangent operator.
In practice, one may either fix a modest iteration budget in large-scale settings or use a residual-based stopping criterion when higher accuracy is required.

\paragraph{Regularization $\lambda$.}
When the tangent Hessian is ill-conditioned or nearly singular, we solve
\[
(H_T+\lambda I)u=b
\]
instead of the unregularized system.
The parameter $\lambda>0$ stabilizes both the log-determinant gradient computation and the final tangent solve.
In practice, $\lambda$ should be chosen small enough not to distort the geometry, yet large enough to prevent numerical instability.

\subsection{Preconditioning strategies}

Preconditioning is important for the efficiency of the stochastic MF-LogDet implementation, since the cost of the method is strongly influenced by the inner Krylov solves.
Natural choices include:
\begin{itemize}
\item diagonal (Jacobi) preconditioners based on local curvature surrogates,
\item block-diagonal approximations when the polynomial structure exhibits weakly coupled variable groups,
\item structure-aware preconditioners derived from sparse polynomial interactions.
\end{itemize}

The most effective choice depends on the algebraic structure of the polynomial model.
In particular, when the monomial support graph is sparse, one may exploit this structure to build preconditioners that preserve the matrix-free character of the method while substantially reducing the number of inner iterations.

\subsection{Parallelization and implementation}

The MF-LogDet procedure is naturally parallelizable at several levels.
First, polynomial, gradient, and Hessian--vector evaluations decompose over monomials and can therefore be parallelized directly.
Second, the Hutchinson probes are independent and may be evaluated concurrently.
Third, the underlying matrix-free kernels are well suited to modern accelerator architectures, since the dominant cost is concentrated in repeated sparse directional contractions rather than dense tensor assembly.

This structure is particularly attractive on GPUs or multicore systems:
the expensive part of the computation lies in repeated Hessian--vector and directional third-derivative kernels, both of which admit efficient batched implementations.

\subsection{Summary}

The MF-LogDet reformulation turns affine normal computation into a sequence of matrix-free operator evaluations and linear-algebraic primitives. Its exact variant is useful for validation and controlled comparison with the original explicit formula, while its stochastic variant provides the scalable implementation relevant to high-dimensional sparse polynomial problems. In both cases, the method avoids explicit third-order tensor assembly and shifts the dominant cost to structured operator applications that are substantially more amenable to large-scale computation. The next section quantifies this advantage through a detailed complexity analysis.

\section{Complexity Analysis and Structural Scalability}
\label{sec:complexity}

In this section we provide a rigorous complexity analysis of the proposed
matrix-free affine normal computation framework.
Our goal is twofold:
(i) to quantify the computational cost in terms of problem structure,
and (ii) to explain why affine normal directions become scalable
in the polynomial setting.

Section~\ref{sec:analytic}--\ref{sec:stochastic} established
a sequence of reductions that transform affine normal computation
into matrix-free second-order primitives.
In particular, all operations reduce to combinations of
Hessian--vector products, linear system solves,
and stochastic trace estimation.

These reductions suggest that affine normal computation should admit
a complexity comparable to matrix-free Newton-type methods.
We now make this statement precise.

\subsection{Computational model}

We formalize the computational model and the structural parameters
that govern the complexity of the algorithm.

\begin{itemize}
\item $d=n+1$: ambient dimension,
\item $m$: number of monomials,
\item $s$: average support size of monomials,
\item $q$: number of stochastic probes,
\item $k$: number of Krylov iterations.
\end{itemize}

\begin{assumption}[Polynomial oracle model]
\label{ass:poly}
The objective function admits a sparse polynomial representation
\[
f(x)=\sum_{\ell=1}^m c_\ell x^{\alpha^{(\ell)}},
\]
with average support size
\[
s := \frac{1}{m}\sum_{\ell=1}^m |\supp(\alpha^{(\ell)})|.
\]
\end{assumption}

\begin{assumption}[Matrix-free oracle]
\label{ass:oracle}
We assume access to gradient evaluation and Hessian--vector products.
Higher-order directional quantities are computed using matrix-free
polynomial evaluations, without explicit formation of third-order tensors.
\end{assumption}

These primitives, together with Krylov subspace methods
and stochastic trace estimation, are sufficient to implement
all steps of Algorithm~\ref{alg:mf-logdet} without explicitly forming
matrices or tensors.

\subsection{Cost of basic primitives}

We first establish the cost of evaluating derivative operators.

\begin{lemma}[Cost of derivative primitives]
\label{lem:primitive-cost}
Under Assumption~\ref{ass:poly}, the following hold:
\begin{align*}
\mathrm{cost}(\nabla f) &= O(ms),\\
\mathrm{cost}(Hv) &= O(ms),\\
\mathrm{cost}(D^3 f[u,v,\cdot]) &= O(ms).
\end{align*}
\end{lemma}

\begin{proof}
Each monomial contributes only on its support set.
The gradient, Hessian--vector product, and third-order contraction
can all be computed via sparse monomial contractions,
with the contribution of the $\ell$-th monomial requiring
$O(|\supp(\alpha^{(\ell)})|)$ operations.
Summing over $m$ monomials yields the result.
\end{proof}

\begin{remark}
Lemma~\ref{lem:primitive-cost} shows that all derivative computations
scale linearly with the sparsity of the polynomial representation,
rather than quadratically or cubically in the ambient dimension.
\end{remark}

\subsection{Cost of stochastic log-determinant gradient}

We now analyze the cost of estimating
\[
a = \nabla_t \log\det(H_T).
\]

\begin{proposition}[Cost per stochastic probe]
\label{prop:probe-cost}
Each Hutchinson probe for estimating $a$
requires:
\begin{itemize}
\item $k$ Hessian--vector products (Krylov iterations),
\item one third-order directional evaluation.
\end{itemize}
Hence the cost per probe is
\[
O(k\,ms).
\]
\end{proposition}

\begin{proof}
Each Krylov iteration requires one application of the operator
$v \mapsto H_T v$, which reduces to a Hessian--vector product.
Thus $k$ iterations cost $O(k\,ms)$ by Lemma~\ref{lem:primitive-cost}.
The third-order contraction also costs $O(ms)$.
Combining these terms yields the total cost
\[
O(k\,ms)+O(ms)=O((k+1)ms)=O(kms).
\]
Here, we neglect the $O(d)$ cost of lifting and projection,
which is dominated by the polynomial Hessian--vector kernel
in the regimes of interest.
\end{proof}

\begin{proposition}[Cost of log-determinant gradient estimation]
\label{prop:logdet-cost}
Using $q$ independent stochastic probes,
the cost of estimating $a = \nabla_t \log\det(H_T)$ is
\[
O(qk\,ms).
\]
\end{proposition}

\begin{proof}
This follows directly from Proposition~\ref{prop:probe-cost}
by summing over $q$ probes.
\end{proof}

\subsection{Cost of solving the tangent linear system}

The affine normal direction requires solving
\[
H_T u = b.
\]

\begin{proposition}[Krylov complexity]
\label{prop:krylov}
Assume that $H_T$ is symmetric positive definite
with condition number $\kappa$.
Then the conjugate gradient method computes an $\varepsilon$-accurate solution
in
\[
k = O\big(\sqrt{\kappa}\log(1/\varepsilon)\big)
\]
iterations.
\end{proposition}
\begin{proof}
The bound is the classical convergence estimate for the conjugate gradient method
applied to a symmetric positive definite system; see, e.g.,
\cite{saad2003,greenbaum1997}.
\end{proof}

\begin{remark}
Each iteration requires one application of the operator
$v \mapsto H_T v$, which is implemented via
Hessian--vector products.
\end{remark}

\begin{remark}
In the indefinite case, one may use MINRES or solve
$(H_T + \lambda I)u=b$
with $\lambda>0$ to ensure stability.
The complexity bounds remain unchanged up to constants.
\end{remark}

\subsection{End-to-end complexity}

We now combine all components to obtain the total complexity.

\begin{theorem}[End-to-end complexity of affine normal computation]
\label{thm:final-complexity}
Under Assumptions~\ref{ass:poly}--\ref{ass:oracle},
the affine normal direction can be computed with computational cost
\[
O\big(q\,\sqrt{\kappa}\,\log(1/\varepsilon)\cdot ms\big),
\]
using a matrix-free implementation, where $\kappa$ denotes the condition number of the tangent linear system
(and of its regularized counterpart when a shift $\lambda I$ is used),
and $\varepsilon$ is the prescribed Krylov solve tolerance.

In particular, up to a stochastic factor $q$,
the computational complexity is of the same order as matrix-free Newton-CG-type methods.
\end{theorem}

\begin{proof}
The computation consists of three steps:

\paragraph{Step 1: Mixed Hessian term.}
Computing
\[
h = T^\top (\nabla^2 f(x)\,\nu)
\]
requires one Hessian--vector product, hence costs $O(ms)$.

\paragraph{Step 2: Log-determinant gradient.}
By Proposition~\ref{prop:logdet-cost},
estimating $a$ requires $O(qk\,ms)$ operations.

\paragraph{Step 3: Linear system solve.}
Solving
\[
H_T u = b
\]
requires $k$ Krylov iterations,
each costing $O(ms)$,
thus totaling $O(k\,ms)$.

\paragraph{Total cost.}
Combining all terms yields
\[
O(qk\,ms).
\]
Substituting
\[
k = O(\sqrt{\kappa}\log(1/\varepsilon))
\]
from Proposition~\ref{prop:krylov}
gives the result.
\end{proof}

\subsection{Error analysis}

We now quantify how approximation errors in the log-determinant gradient and inexact linear solves propagate to the final affine normal direction. Let
\[
A := H_T+\lambda I,
\qquad
\beta := \frac{\|g\|}{d+1},
\]
and consider the exact tangent coefficient vector
\[
u = A^{-1}(h-\beta a),
\]
where
\[
a = \nabla_t \log\det(H_T).
\]
Suppose that $\hat a$ is an approximation of $a$, and that $\hat u$ is obtained from an inexact linear solve with residual
\[
r_{\mathrm{Krylov}}
:=
A\hat u-(h-\beta \hat a).
\]

\begin{proposition}[Error propagation to the tangent coefficient vector]
\label{prop:error-u}
Assume that $A=H_T+\lambda I$ is invertible. Then
\begin{equation}
\label{eq:error-u}
\|\hat u-u\|
\le
\|A^{-1}\|
\Big(
\beta \|\hat a-a\| + \|r_{\mathrm{Krylov}}\|
\Big).
\end{equation}
In particular, if $A$ is symmetric positive definite, then
\begin{equation}
\label{eq:error-u-spd}
\|\hat u-u\|
\le
\frac{1}{\lambda_{\min}(A)}
\Big(
\beta \|\hat a-a\| + \|r_{\mathrm{Krylov}}\|
\Big).
\end{equation}
\end{proposition}

\begin{proof}
By definition of the exact and approximate systems,
\[
Au = h-\beta a,
\qquad
A\hat u = h-\beta \hat a + r_{\mathrm{Krylov}}.
\]
Subtracting the two identities gives
\[
A(\hat u-u)
=
-\beta(\hat a-a)+r_{\mathrm{Krylov}}.
\]
Since $A$ is invertible,
\[
\hat u-u
=
A^{-1}\bigl(-\beta(\hat a-a)+r_{\mathrm{Krylov}}\bigr).
\]
Taking norms and using the triangle inequality,
\[
\|\hat u-u\|
\le
\|A^{-1}\|
\Big(
\beta \|\hat a-a\| + \|r_{\mathrm{Krylov}}\|
\Big),
\]
which proves \eqref{eq:error-u}. If, in addition, $A$ is symmetric positive definite, then
\[
\|A^{-1}\|=\frac{1}{\lambda_{\min}(A)},
\]
and \eqref{eq:error-u-spd} follows immediately.
\end{proof}

The affine normal direction in ambient coordinates is
\[
d_{\mathrm{AN}} = Tu-\nu,
\]
where the columns of $T$ form an orthonormal basis of the tangent space and $\nu$ is the unit Euclidean normal.
Using the same $T$ and $\nu$, define the approximate direction
\[
\hat d_{\mathrm{AN}} = T\hat u-\nu.
\]

\begin{corollary}[Error propagation to the affine normal direction]
\label{cor:error-d}
Under the assumptions of Proposition~\ref{prop:error-u},
\begin{equation}
\label{eq:error-d}
\|\hat d_{\mathrm{AN}}-d_{\mathrm{AN}}\|
=
\|T(\hat u-u)\|
=
\|\hat u-u\|
\le
\|A^{-1}\|
\Big(
\beta \|\hat a-a\| + \|r_{\mathrm{Krylov}}\|
\Big).
\end{equation}
If $A$ is symmetric positive definite, then
\begin{equation}
\label{eq:error-d-spd}
\|\hat d_{\mathrm{AN}}-d_{\mathrm{AN}}\|
\le
\frac{1}{\lambda_{\min}(A)}
\Big(
\beta \|\hat a-a\| + \|r_{\mathrm{Krylov}}\|
\Big).
\end{equation}
\end{corollary}

\begin{proof}
Since
\[
\hat d_{\mathrm{AN}}-d_{\mathrm{AN}}
=
T(\hat u-u),
\]
and the columns of $T$ are orthonormal, we have
\[
\|T(\hat u-u)\|=\|\hat u-u\|.
\]
The estimate therefore follows directly from Proposition~\ref{prop:error-u}.
\end{proof}

In many numerical experiments, one compares normalized directions rather than the raw affine normal vectors.
The next estimate provides the corresponding perturbation bound.

\begin{corollary}[Normalized direction error]
\label{cor:error-normalized-d}
Assume that $d_{\mathrm{AN}}\neq 0$ and
\[
\|\hat d_{\mathrm{AN}}-d_{\mathrm{AN}}\|
\le
\frac{1}{2}\|d_{\mathrm{AN}}\|.
\]
Then
\begin{equation}
\label{eq:error-normalized-d}
\left\|
\frac{\hat d_{\mathrm{AN}}}{\|\hat d_{\mathrm{AN}}\|}
-
\frac{d_{\mathrm{AN}}}{\|d_{\mathrm{AN}}\|}
\right\|
\le
\frac{2}{\|d_{\mathrm{AN}}\|}
\|\hat d_{\mathrm{AN}}-d_{\mathrm{AN}}\|.
\end{equation}
Consequently,
\begin{equation}
\label{eq:error-normalized-d-full}
\left\|
\frac{\hat d_{\mathrm{AN}}}{\|\hat d_{\mathrm{AN}}\|}
-
\frac{d_{\mathrm{AN}}}{\|d_{\mathrm{AN}}\|}
\right\|
\le
\frac{2\|A^{-1}\|}{\|d_{\mathrm{AN}}\|}
\Big(
\beta \|\hat a-a\| + \|r_{\mathrm{Krylov}}\|
\Big).
\end{equation}
\end{corollary}

\begin{proof}
The estimate \eqref{eq:error-normalized-d} is the standard perturbation bound for normalization of nonzero vectors. Applying Corollary~\ref{cor:error-d} then yields \eqref{eq:error-normalized-d-full}.
\end{proof}

The preceding estimates make explicit the two sources of inexactness in affine normal computation: the approximation error in the log-determinant gradient and the residual of the Krylov linear solve. In particular, they show that, provided the regularized tangent operator remains well conditioned, both errors propagate in a stable and essentially linear manner to the tangent coefficient vector and hence to the final affine normal direction. This justifies the use of stochastic trace estimation and inexact linear solvers in the scalable implementation.

\subsection{Structural interpretation}

A key conceptual consequence of the analysis is that
affine normal directions are fundamentally second-order objects. Importantly, no explicit third-order tensor is formed.
All higher-order contributions appear only through
directional contractions, whose cost matches that of
Hessian--vector products in the polynomial setting.

\begin{theorem}[Second-order reduction of affine normal computation]
\label{thm:second-order}
Affine normal directions can be computed using only:
\begin{itemize}
\item Hessian--vector products,
\item linear system solves,
\item stochastic trace estimation.
\end{itemize}
No explicit third-order tensor needs to be formed.
\end{theorem}
\begin{proof}
By Lemma~\ref{lem:logdet-sec2}, the third-order contraction term
in the analytic affine normal formula can be rewritten as
\[
f^{pq}f_{pqi}=\partial_{i}\log\det(H_T).
\]
Hence the affine normal depends on the tangent Hessian only through:
(i) Hessian--vector products used to apply the matrix-free operator $H_T$,
(ii) linear system solves involving $H_T$,
and (iii) stochastic trace estimation of the log-determinant gradient.
No explicit third-order tensor needs to be formed.
\end{proof}

\begin{remark}
Although the affine normal is classically defined through third-order differential geometry, the log-determinant identity shows that its computation can be organized around matrix-free second-order linear-algebraic primitives. In the polynomial setting, the remaining higher-order information enters only through structured directional contractions, which are computable at the same sparsity order as Hessian--vector products.
\end{remark}

\begin{remark}[Interpretation]
The affine normal direction can be viewed as a curvature-corrected
Newton direction, where the correction term
$\nabla \log\det(H_T)$ accounts for local volume distortion
of level sets.
\end{remark}

\subsection{Scalability with respect to dimension}

The complexity depends on $ms$, not on $d^2$ or $d^3$.

\begin{itemize}
\item If $f$ is sparse, $ms \ll d^2$,
\item If $f$ has bounded support size and the number of monomials grows linearly with dimension, then $ms = O(d)$.
\end{itemize}

Thus affine normal computation scales linearly in dimension
for structured polynomial objectives.

\begin{corollary}\label{cor:linear-complexity-O(d)}
If the support size is uniformly bounded, i.e., $s=O(1)$,
the number of monomials satisfies $m=O(d)$,
and the parameters $q$, $\kappa$, and $\varepsilon$ are treated as constants,
then affine normal computation has linear complexity $O(d)$.
\end{corollary}
\begin{proof}
If $s=O(1)$ and the number of monomials scales linearly with dimension,
i.e., $m=O(d)$, then
\[
ms = O(d).
\]
Substituting into Theorem~\ref{thm:final-complexity}
gives the result.
\end{proof}

The preceding corollary addresses the common regime in which sparsity remains controlled while the ambient dimension grows. It is also useful to state the complexity directly in terms of the algebraic sparsity of the polynomial representation itself.
\begin{proposition}[Complexity in terms of sparsity]
\label{prop:complexity-nnz}
The computational cost of MF-LogDet affine normal computation is linear in the total sparsity of the polynomial representation:
\[
\mathrm{cost} = O(qk \cdot \mathrm{nnz}),
\qquad
\mathrm{nnz}:=\sum_{\ell=1}^m |\supp(\alpha^{(\ell)})|.
\]
\end{proposition}

\begin{proof}
By Proposition~\ref{prop:probe-cost}, the cost of one stochastic probe is proportional to the cost of a Hessian--vector product together with one directional third-order contraction. By Proposition~\ref{prop:Hv-kernel} and Proposition~\ref{prop:d3-kernel}, each of these primitives can be evaluated in
\[
O\!\Big(\sum_{\ell=1}^m |\supp(\alpha^{(\ell)})|\Big)
=
O(\mathrm{nnz})
\]
operations. Summing over $q$ probes and $k$ Krylov steps yields
\[
\mathrm{cost}=O(qk\cdot \mathrm{nnz}),
\]
as claimed.
\end{proof}

This shows that the computational complexity is governed by the algebraic sparsity
of the objective, rather than the ambient dimension itself.

\subsection{Discussion}

We summarize the main conclusions:

\begin{itemize}
\item[(i)] All third-order terms are eliminated through log-determinant reduction.
\item[(ii)] The algorithm is fully matrix-free.
\item[(iii)] Polynomial structure enables linear scaling.
\end{itemize}

These results resolve the main computational barrier
to applying affine-normal-based methods in high-dimensional polynomial optimization.

\section{Numerical Experiments}
\label{sec:experiments}

In this section we evaluate the proposed matrix-free log-determinant formulation (MF-LogDet) for affine normal computation. Our goal is not to benchmark the full optimization performance of YAND, but rather to directly validate the main computational claims of this paper.

More precisely, the experiments are organized around four questions:
\begin{itemize}
\item[(i)] Does MF-LogDet reproduce the same affine normal direction as the original explicit autodifferentiation (AD)-based computation?
\item[(ii)] How does stochastic trace approximation affect the accuracy and cost of the computed affine normal direction?
\item[(iii)] Does the runtime of MF-LogDet scale nearly linearly with the ambient dimension in the sparse polynomial regime?
\item[(iv)] Does the runtime of MF-LogDet scale nearly linearly with the structural sparsity level of the polynomial representation?
\end{itemize}

Accordingly, Section~\ref{subsec:mf-logdet-accuracy-speed} validates the exactness of the log-determinant reformulation, Section~\ref{subsec:stochastic-trace-accuracy} studies the stochastic accuracy--cost tradeoff, Section~\ref{subsec:mf-logdet-near-linear} examines scaling with respect to dimension, and Section~\ref{subsec:sparsity-scaling} examines scaling with respect to sparsity.

Unless otherwise stated, all experiments below measure the cost of computing a \emph{single affine normal direction}, rather than the cost of the outer optimization loop. This isolates the computational effect of the proposed reformulation from outer-iteration effects such as line search or step acceptance.

Throughout this section, we compare the following two implementations:
\begin{itemize}
\item \textbf{AD}: the original explicit autodifferentiation-based affine normal computation using the direct third-order formula (see \cite{AND2026});
\item \textbf{MF-LogDet}: the proposed matrix-free log-determinant reformulation developed in this paper.
\end{itemize}

To compare directional accuracy, we use the normalized direction error
\[
\bigl\|\widehat d_{\mathrm{MF\mbox{-}LogDet}}-\widehat d_{\mathrm{AD}}\bigr\|,
\qquad
\widehat d:=\frac{d}{\|d\|},
\]
together with the corresponding angular error.
To compare computational efficiency, we report both the absolute runtime and the speedup ratio
\[
\mathrm{speedup}(d)
=
\frac{T_{\mathrm{AD}}(d)}{T_{\mathrm{MF\mbox{-}LogDet}}(d)}.
\]
Values larger than $1$ indicate that MF-LogDet is faster than AD.
All reported runtimes are averages over repeated affine-normal evaluations.

\subsection{Accuracy and acceleration of affine normal computation}
\label{subsec:mf-logdet-accuracy-speed}

We first compare AD and MF-LogDet on a structured sparse quartic polynomial family in dimensions
\[
d=3,4,\dots,20.
\]
The purpose of this experiment is to compare the two affine normal formulas as directly as possible,
without stochastic trace approximation or iterative solver effects.

For each dimension $d$, we consider a deterministic sparse quartic polynomial of the form
\[
f(x)
=
\sum_{i=1}^d a_i x_i^4
+
\sum_{i=1}^{d-1} b_i x_i^2x_{i+1}^2
+
\sum_{i=1}^{d-2} c_i x_i^3x_{i+2}
+
\sum_{k=1}^{\lfloor d/3\rfloor} \gamma_k x_{3k-2}^2 x_{3k-1} x_{3k},
\]
with fixed deterministic coefficients $\{a_i,b_i,c_i,\gamma_k\}$.
This defines a structured sparse quartic test family in which both AD and MF-LogDet can be evaluated reliably across dimensions.
Moreover, the sparsity pattern remains controlled, and the average monomial support size stays bounded.

This construction yields a nontrivial but controlled test family in which both AD and MF-LogDet can be evaluated reliably across dimensions.
For each dimension, affine normal directions are computed at several deterministic sample points generated from smooth sinusoidal patterns.
In this experiment, MF-LogDet is run in an exact/direct mode, so that the comparison reflects only the difference between the explicit third-order formula and the proposed log-determinant reformulation itself.

\paragraph{Accuracy.}
Figure~\ref{fig:mf-logdet-error} shows that MF-LogDet reproduces the affine normal direction computed by AD to near machine precision across the entire tested range.
The normalized direction error remains around $10^{-9}$,
and the angular error stays on the order of $10^{-6}$ degrees or below.
This is a direct numerical validation of the central identity
\[
f^{pq}f_{pqi}=\partial_i\log\det(H_T),
\]
which underlies the proposed reformulation.
Numerically, the experiment shows that replacing the explicit third-order contraction by the log-determinant gradient does not degrade the computed affine normal direction.

\paragraph{Speedup and runtime crossover.}
Figure~\ref{fig:mf-logdet-speedup} reports the speedup ratio
$T_{\mathrm{AD}}/T_{\mathrm{MF\mbox{-}LogDet}}$ as the dimension increases.
In very low dimensions, the explicit AD-based computation is competitive, and can even be faster, since the derivative tensors remain small and inexpensive to evaluate.
However, as $d$ increases, the computational cost of AD grows rapidly, whereas MF-LogDet scales much more favorably.
In our experiment, the crossover occurs around $d\approx 11$, and by $d=20$ the speedup reaches approximately $11.9\times$.

Figure~\ref{fig:mf-logdet-absolute-time} complements this comparison by showing the absolute runtime per affine normal evaluation for both methods.
This figure makes clear that the advantage of MF-LogDet is not merely a constant-factor artifact:
the two methods exhibit genuinely different scaling trends as the dimension increases.

\paragraph{Interpretation.}
Taken together, Figures~\ref{fig:mf-logdet-error}--\ref{fig:mf-logdet-absolute-time}
support the first main computational claim of this paper:
MF-LogDet does not merely provide a cheaper approximation of affine normal directions,
but reproduces essentially the same direction while significantly reducing the computational burden once the dimension becomes moderately large.
This reflects the structural transition from explicit tensor-based evaluation to matrix-free operator-based computation.

\begin{figure}[t]
    \centering
    \includegraphics[width=0.82\textwidth]{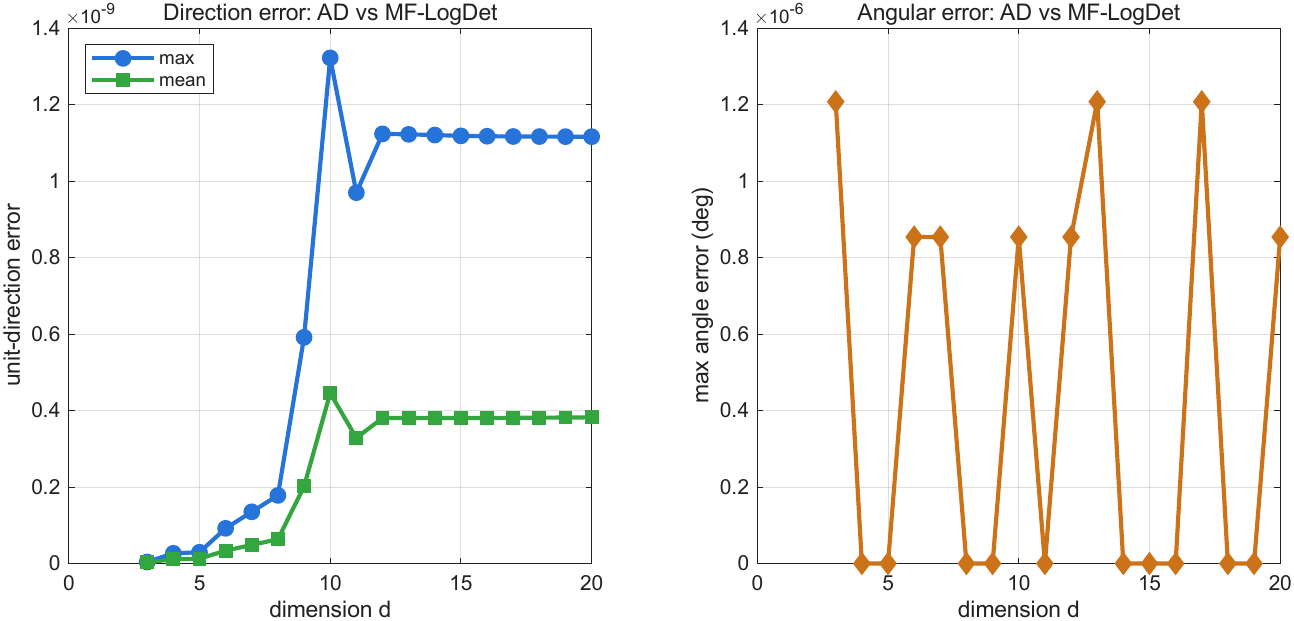}
    \caption{Accuracy of MF-LogDet relative to the original AD-based affine normal computation.
    The normalized direction error remains around $10^{-9}$, while the angular error stays near machine precision,
    showing that the proposed reformulation reproduces the same affine normal direction very accurately.}
    \label{fig:mf-logdet-error}
\end{figure}

\begin{figure}[t]
    \centering
    \includegraphics[width=0.72\textwidth]{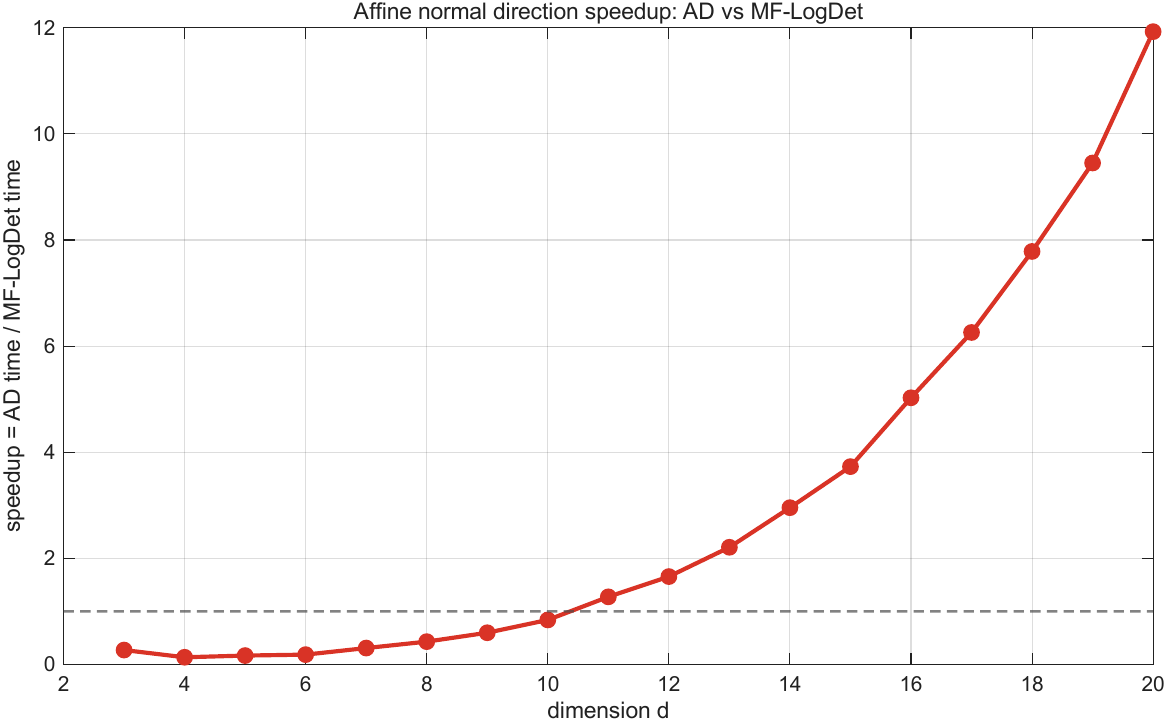}
    \caption{Speedup of affine normal computation across dimensions, defined by
    $T_{\mathrm{AD}}/T_{\mathrm{MF\mbox{-}LogDet}}$.
    Values above $1$ indicate that MF-LogDet is faster than AD.
    The crossover occurs around $d\approx 11$, and by $d=20$ the speedup reaches about $11.9\times$.}
    \label{fig:mf-logdet-speedup}
\end{figure}

\begin{figure}[t]
    \centering
    \includegraphics[width=0.72\textwidth]{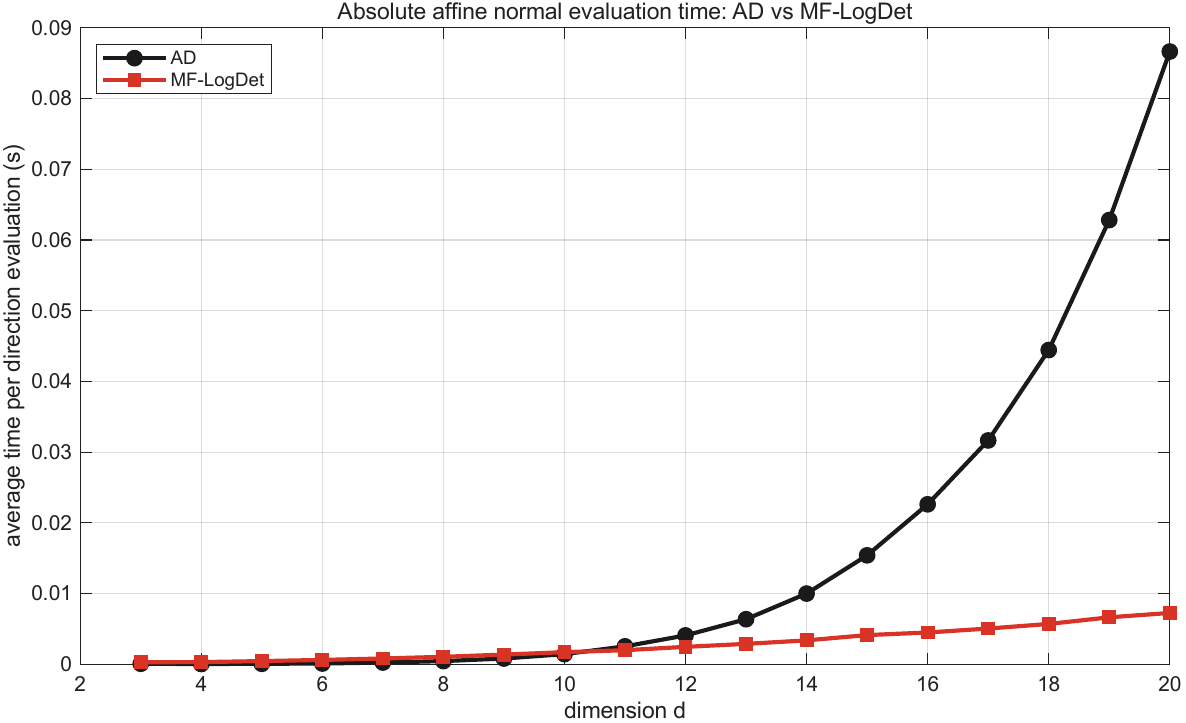}
    \caption{Average runtime per affine normal evaluation for AD and MF-LogDet as a function of the dimension $d$.
    While AD is competitive in very low dimensions, its runtime grows much more rapidly.
    MF-LogDet becomes more efficient once the dimension is moderately large.}
    \label{fig:mf-logdet-absolute-time}
\end{figure}

\subsection{Effect of stochastic trace approximation on affine normal accuracy}
\label{subsec:stochastic-trace-accuracy}

The previous experiment validated the log-determinant reformulation itself by comparing MF-LogDet with AD in an exact/direct setting.
We now examine a complementary question:
when the tangent log-determinant gradient
\[
a=\nabla_t \log\det(H_T)
\]
is approximated by stochastic trace estimation, how accurately is the resulting affine normal direction preserved, and how does this approximation affect the computational cost?

\paragraph{Experimental setup.}
We consider a structured sparse quartic polynomial family in dimensions
\[
d\in\{10,20,40\}.
\]
For each dimension, affine normal directions are evaluated at several deterministic sample points generated from smooth sinusoidal patterns.
At each point, we first compute the exact tangent log-determinant gradient and the corresponding affine normal direction by direct linear algebra in the tangent space.
We then replace the exact trace computation by the Hutchinson estimator with probe counts
\[
q\in\{2,5,10,20,50,100\},
\]
while keeping the remaining direction assembly exact.
Thus this experiment isolates the effect of stochastic trace approximation alone: the only inexact component is the Hutchinson-based approximation of the tangent log-determinant gradient, while the remaining tangent-space linear algebra is performed deterministically.
All runtimes are measured after a short warm-up phase and averaged over repeated evaluations in order to reduce timing noise.
In other words, the reference direction in this experiment is the exact/direct MF-LogDet computation, not the original AD formula, so that the effect of stochastic trace approximation can be isolated cleanly.

\paragraph{Measured quantities.}
We report:
\begin{itemize}
\item the relative error in the tangent log-determinant gradient,
\[
\frac{\|a_q-a_{\mathrm{exact}}\|}{\|a_{\mathrm{exact}}\|},
\]
\item the normalized affine normal direction error,
\[
\bigl\|\widehat d_q-\widehat d_{\mathrm{exact}}\bigr\|,
\qquad
\widehat d=\frac{d}{\|d\|},
\]
\item the corresponding angular error,
\item the runtime of the stochastic trace approximation,
\item and the runtime ratio between stochastic and exact tangent log-determinant evaluation.
\end{itemize}

\paragraph{Accuracy.}
Figure~\ref{fig:stochastic-trace-accuracy} shows that the affine normal direction becomes progressively more accurate as the number of Hutchinson probes increases.
For instance, in dimension $d=10$, the mean normalized direction error decreases from approximately $2.17\times 10^{-2}$ at $q=2$ to about $6.70\times 10^{-3}$ at $q=50$.
In dimension $d=20$, using $q=50$ probes reduces the mean normalized direction error to approximately $3.12\times 10^{-3}$, with a mean angular error of about $0.18^\circ$.
In dimension $d=40$, using $q=100$ probes yields a mean normalized direction error of approximately $1.42\times 10^{-3}$ and a mean angular error of about $0.08^\circ$.
Thus, although very small probe counts can introduce noticeable fluctuations, moderate probe counts already recover the affine normal direction with good accuracy.

\paragraph{Runtime and comparison with exact trace evaluation.}
Figure~\ref{fig:stochastic-trace-runtime} shows that the runtime of the stochastic approximation grows approximately linearly with $q$, as expected from the trace-estimation model.
For example, in dimension $d=10$, the average runtime increases from about $1.47\times 10^{-4}$ seconds at $q=2$ to about $5.99\times 10^{-3}$ seconds at $q=100$.

It is also important to compare this cost with the exact tangent log-determinant computation used in the first experiment.
Figure~\ref{fig:stochastic-trace-runtime-vs-exact} shows that the advantage of stochastic trace approximation depends on both the dimension and the probe count.
At $d=10$, the Hutchinson estimator is cheaper than the exact computation only for very small probe counts, and becomes more expensive once $q$ is moderately large.
At $d=20$, the crossover occurs around $q\approx 20$.
At $d=40$, however, the stochastic approximation remains substantially cheaper than the exact baseline for small and moderate probe counts, and only becomes more expensive when $q$ is large.
This confirms the intended role of stochastic trace estimation: it is most beneficial in larger-scale regimes where exact tangent log-determinant evaluation is itself costly.

\paragraph{Representative tradeoff.}
Table~\ref{tab:stochastic-trace-accuracy} provides representative numerical values for the resulting accuracy--cost tradeoff.
In all tested dimensions, increasing the probe count $q$ improves the affine normal direction accuracy, while the runtime grows approximately proportionally to $q$.
At the same time, the relative cost with respect to exact tangent log-determinant evaluation depends strongly on the dimension.
For $d=10$, the stochastic approximation is cheaper only for very small probe counts and becomes more expensive once $q$ is moderately large.
For $d=20$, the crossover occurs around $q\approx 20$.
For $d=40$, the stochastic approximation remains cheaper than the exact baseline even at $q=20$, and only becomes more expensive when $q$ is large.
The table therefore complements Figures~\ref{fig:stochastic-trace-accuracy}--\ref{fig:stochastic-trace-runtime-vs-exact} by making the practical accuracy--cost tradeoff explicit.

\paragraph{Interpretation.}
This experiment fills the gap between the exact/direct validation of MF-LogDet and the large-scale complexity experiments.
It shows that the stochastic trace approximation is not merely asymptotically appealing, but can recover the affine normal direction to useful accuracy with moderate probe counts.
In particular, the results support the practical claim that small-to-moderate values of $q$ already provide a reasonable balance between accuracy and efficiency, while the computational benefit of stochastic trace estimation becomes more pronounced as the dimension increases.

\begin{figure}[t]
    \centering
    \includegraphics[width=0.9\textwidth]{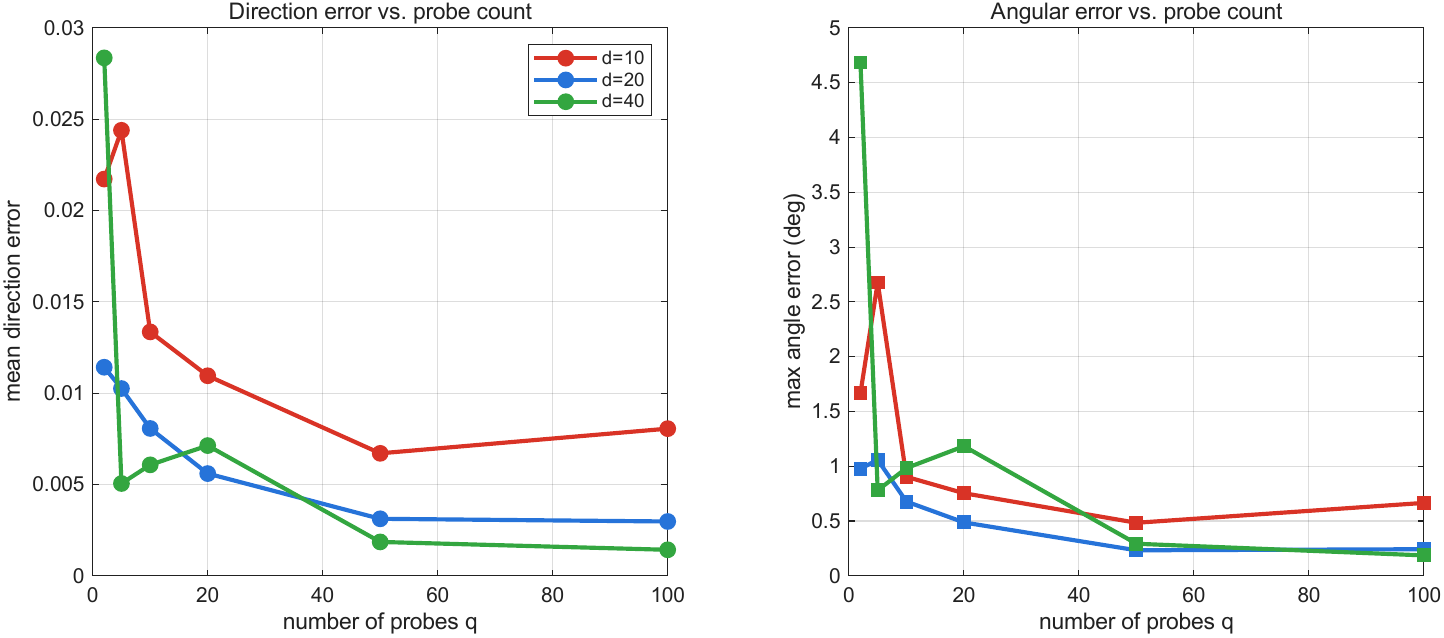}
    \caption{Effect of the Hutchinson probe count $q$ on affine normal accuracy.
    The left panel shows the mean normalized direction error, and the right panel shows the maximum angular error.
    In all tested dimensions, increasing $q$ improves the accuracy of the recovered affine normal direction.}
    \label{fig:stochastic-trace-accuracy}
\end{figure}

\begin{figure}[t]
    \centering
    \includegraphics[width=0.72\textwidth]{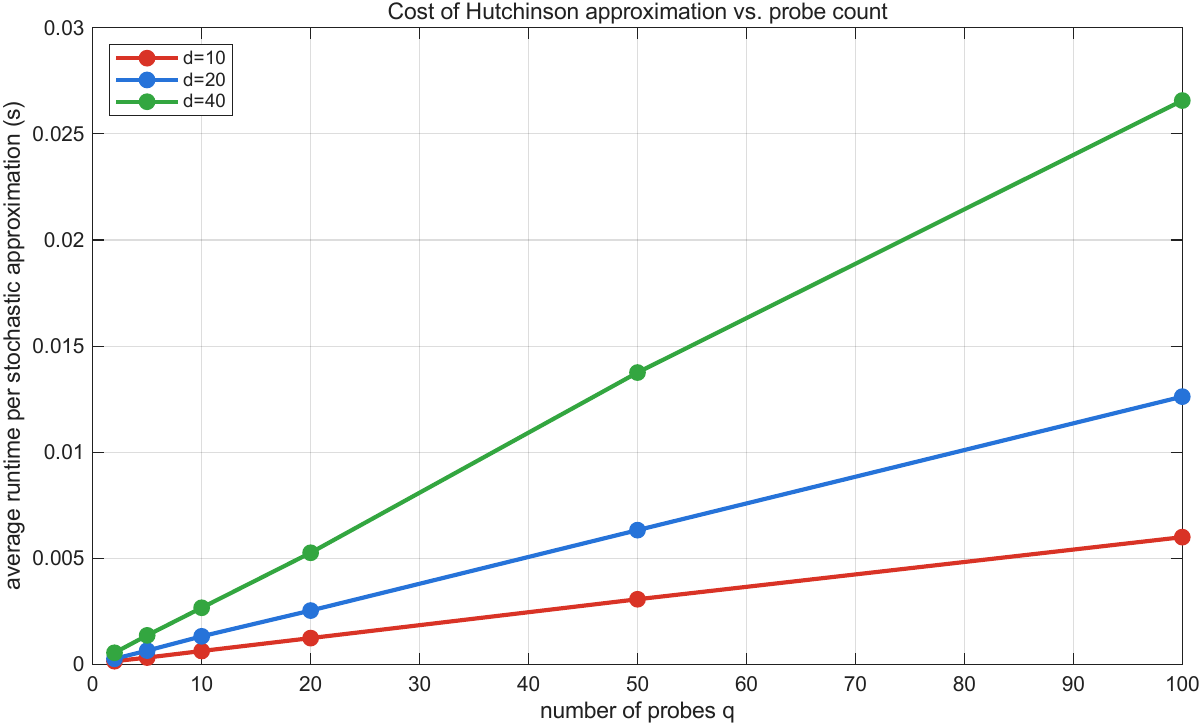}
    \caption{Runtime of the stochastic trace approximation as a function of the number of Hutchinson probes $q$.
    After warm-up and repeated averaging, the cost grows approximately linearly with $q$, consistent with the theoretical complexity model.}
    \label{fig:stochastic-trace-runtime}
\end{figure}

\begin{figure}[t]
    \centering
    \includegraphics[width=0.9\textwidth]{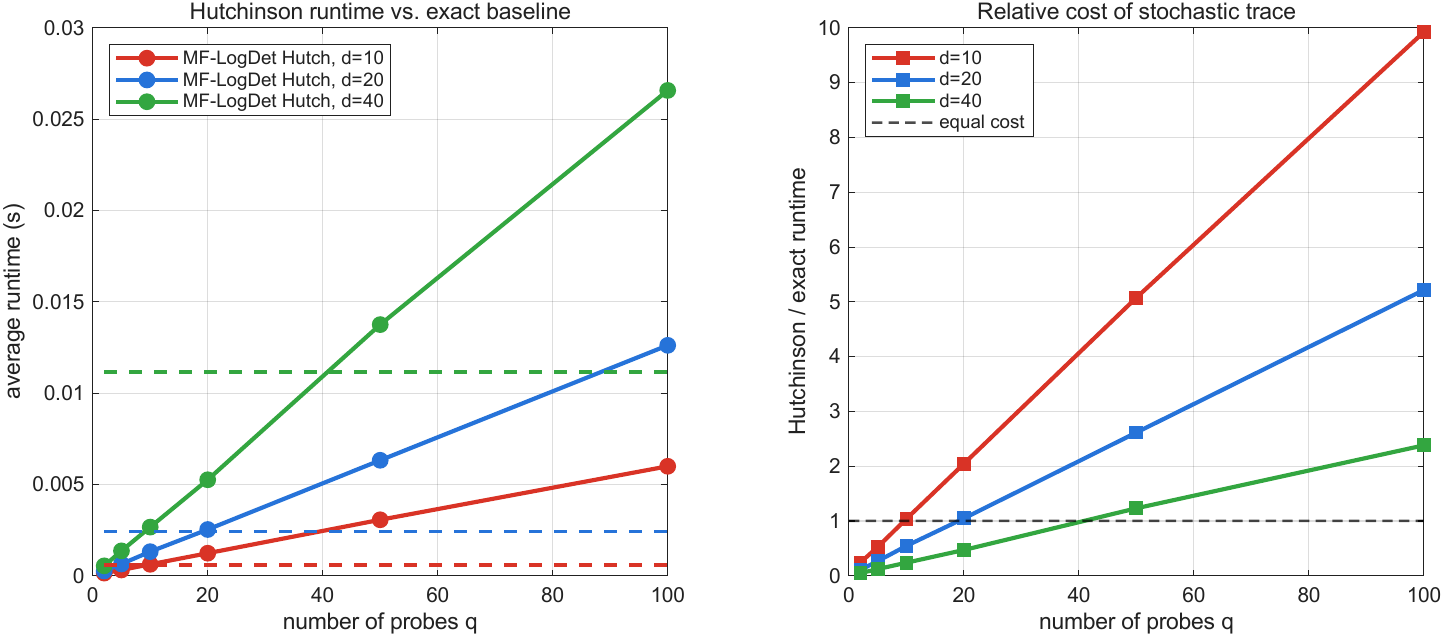}
    \caption{Runtime comparison between stochastic Hutchinson trace approximation and exact tangent log-determinant evaluation.
    The left panel shows absolute runtime, where dashed horizontal lines denote the exact baseline for each dimension.
    The right panel shows the runtime ratio between Hutchinson and exact evaluation.
    Small probe counts are substantially cheaper than exact evaluation, especially in higher dimensions, whereas large probe counts eventually remove this advantage.}
    \label{fig:stochastic-trace-runtime-vs-exact}
\end{figure}

\begin{table}[t]
\centering
\caption{Representative accuracy--cost tradeoff for stochastic trace approximation in MF-LogDet affine normal computation. The last column reports the runtime ratio between Hutchinson approximation and exact tangent log-determinant evaluation.}
\label{tab:stochastic-trace-accuracy}
\begin{tabular}{cccccc}
\toprule
$d$ & $q$ & Mean direction error & Mean angle error (deg) & Mean runtime (s) & Hutch / exact \\
\midrule
10 & 2   & $2.17\times 10^{-2}$ & $1.24$ & $1.47\times 10^{-4}$ & $0.24$ \\
10 & 20  & $1.10\times 10^{-2}$ & $0.63$ & $1.24\times 10^{-3}$ & $2.05$ \\
10 & 100 & $8.06\times 10^{-3}$ & $0.46$ & $5.99\times 10^{-3}$ & $9.93$ \\
20 & 2   & $1.14\times 10^{-2}$ & $0.65$ & $2.61\times 10^{-4}$ & $0.11$ \\
20 & 20  & $5.60\times 10^{-3}$ & $0.32$ & $2.53\times 10^{-3}$ & $1.05$ \\
20 & 100 & $2.97\times 10^{-3}$ & $0.17$ & $1.26\times 10^{-2}$ & $5.22$ \\
40 & 2   & $2.84\times 10^{-2}$ & $1.63$ & $5.44\times 10^{-4}$ & $0.05$ \\
40 & 20  & $7.13\times 10^{-3}$ & $0.41$ & $5.26\times 10^{-3}$ & $0.47$ \\
40 & 100 & $1.42\times 10^{-3}$ & $0.08$ & $2.66\times 10^{-2}$ & $2.38$ \\
\bottomrule
\end{tabular}
\end{table}

\subsection{Empirical near-linear complexity in dimension}
\label{subsec:mf-logdet-near-linear}

We next test the large-scale complexity prediction established in Section~\ref{sec:complexity}.
The theory predicts that, in the sparse polynomial regime, the cost of MF-LogDet affine normal computation should scale essentially linearly with the ambient dimension,
provided that the number of monomials grows linearly and the average support size remains bounded.

To match this regime, we consider dimensions
\[
d\in\{50,100,200,400,800\},
\]
and generate sparse polynomial objectives of the form
\[
f(x)=\sum_{\ell=1}^{m}c_\ell x^{\alpha^{(\ell)}},
\qquad x\in\mathbb{R}^d,
\]
with the following structure:
\begin{itemize}
\item the number of monomials is set to $m=10d$;
\item the average support size remains bounded (approximately $3.5$ in our experiment);
\item a small quartic stabilization term is included to ensure numerical robustness.
\end{itemize}
For MF-LogDet, we fix the Hutchinson probe count at $q=2$, use PCG with a maximum of $k=5$ iterations, and set the regularization parameter to $\lambda=10^{-6}$.
Thus the experiment is designed to match the structural regime under which the theory predicts near-linear complexity.

For each dimension, we measure the average runtime of one MF-LogDet affine normal evaluation and fit the empirical scaling law
\[
T(d)\approx C d^\beta
\]
by linear regression in log-log scale.

\paragraph{Runtime scaling.}
Figure~\ref{fig:mf-logdet-near-linear} shows the runtime behavior in both the original and log-log scales.
The fitted exponent is
\[
\beta=1.0107,
\]
which is extremely close to the ideal linear value $\beta=1$.
This provides strong empirical support for the near-linear complexity predicted by the theory.

\paragraph{Operator-level stability.}
Figure~\ref{fig:mf-logdet-near-linear-counts} and Table~\ref{tab:mf-logdet-near-linear-full}
show that the internal operator counts remain essentially constant across dimensions:
the number of Hessian--vector products per affine normal evaluation stays around $329$--$330$,
the number of third-directional evaluations is fixed at $2$,
and the Krylov workload remains effectively constant at about $300$ inner iterations.
This is important because it shows that the observed runtime growth is driven mainly by the increase of problem size itself,
rather than by deterioration of the inner numerical procedure.

\paragraph{Interpretation.}
These results provide strong empirical support for the dimension-scaling prediction of Section~\ref{sec:complexity}: under bounded support size and linearly growing monomial count, the runtime of MF-LogDet affine normal computation grows essentially linearly with the ambient dimension. The experiment therefore validates not only the observed empirical law, but also the structural mechanism behind Theorem~\ref{thm:final-complexity} and Corollary~\ref{cor:linear-complexity-O(d)}.

\begin{figure}[t]
    \centering
    \begin{subfigure}[t]{0.48\textwidth}
        \centering
        \includegraphics[width=\textwidth]{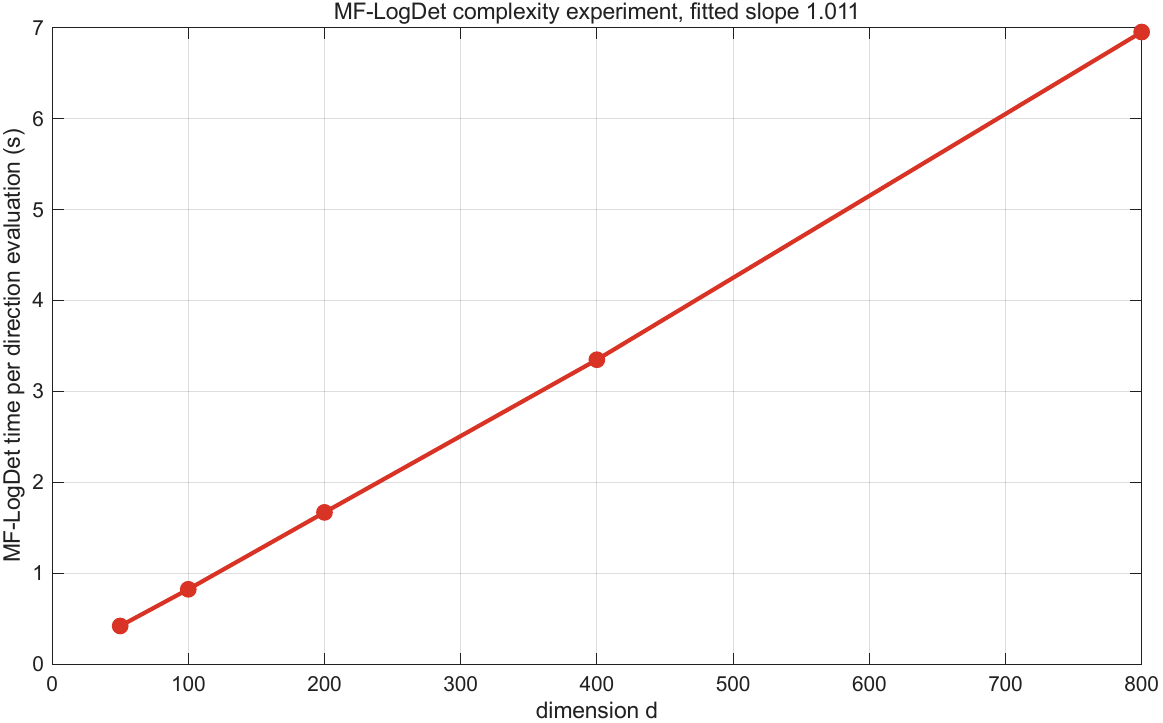}
        \caption{Average runtime per MF-LogDet affine normal evaluation as a function of the dimension $d$.}
        \label{fig:mf-logdet-near-linear-time}
    \end{subfigure}
    \hfill
    \begin{subfigure}[t]{0.48\textwidth}
        \centering
        \includegraphics[width=\textwidth]{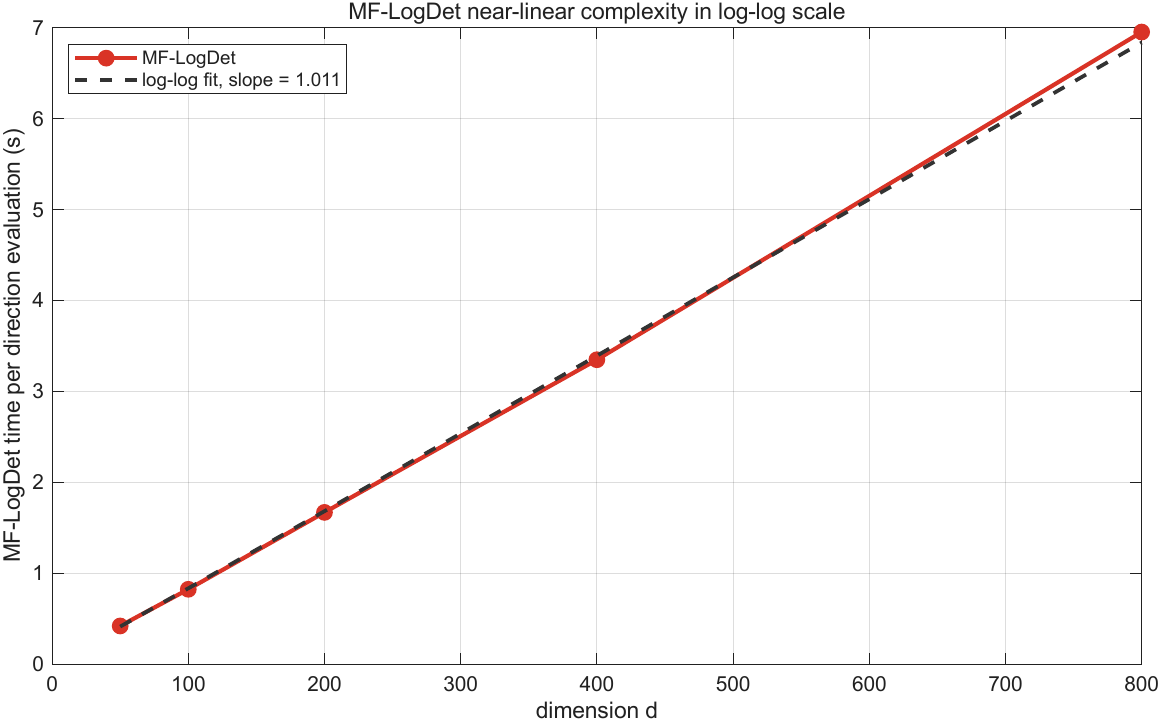}
        \caption{Log-log plot of the same runtime data. The fitted slope is $\beta=1.0107$.}
        \label{fig:mf-logdet-near-linear-loglog}
    \end{subfigure}
    \caption{Empirical near-linear complexity of MF-LogDet affine normal computation in the sparse polynomial regime. The left panel shows the runtime trend in the original scale, while the right panel confirms near-linear behavior through a log-log fit.}
    \label{fig:mf-logdet-near-linear}
\end{figure}

\begin{figure}[t]
    \centering
    \includegraphics[width=0.85\textwidth]{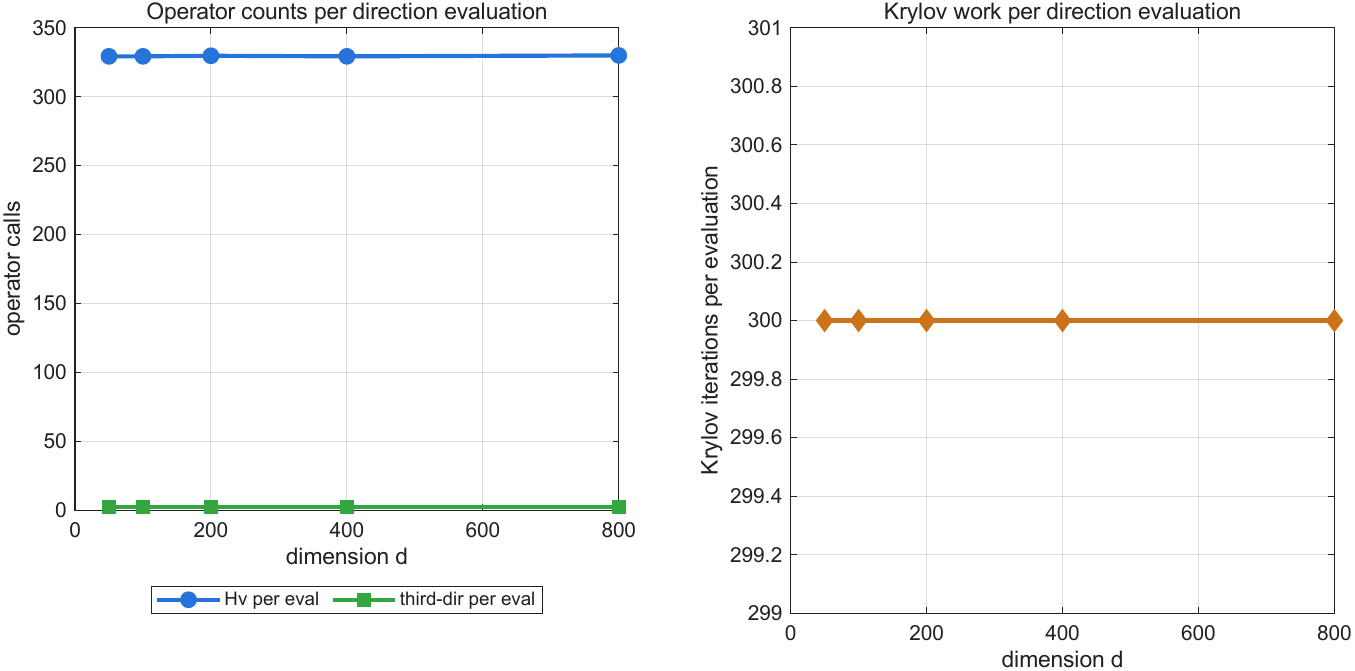}
    \caption{Operator counts per MF-LogDet affine normal evaluation across dimensions.
The Hessian--vector count, directional third-derivative count, and Krylov workload remain essentially constant, indicating that the observed runtime growth in Figure~\ref{fig:mf-logdet-near-linear} is driven primarily by problem size rather than by deterioration of the inner iterative procedures.}
    \label{fig:mf-logdet-near-linear-counts}
\end{figure}

\begin{table}[t]
\centering
\caption{Empirical near-linear complexity of MF-LogDet affine normal computation in the sparse polynomial regime.
The polynomial family satisfies $m=10d$ with bounded average support size.
The observed runtime follows a log-log slope of approximately $1.0107$, consistent with near-linear scaling.}
\label{tab:mf-logdet-near-linear-full}
\begin{tabular}{cccccccc}
\toprule
$d$ & $m$ & Avg. support & Time / eval (s) & Hv / eval & Third / eval & Krylov / eval & Log-log fit \\
\midrule
50  & 550  & 3.6291 & 0.4218 & 329.33 & 2 & 300 & 0.4148 \\
100 & 1100 & 3.4909 & 0.8253 & 329.33 & 2 & 300 & 0.8359 \\
200 & 2200 & 3.5682 & 1.6710 & 329.75 & 2 & 300 & 1.6843 \\
400 & 4400 & 3.5541 & 3.3499 & 329.33 & 2 & 300 & 3.3936 \\
800 & 8800 & 3.5578 & 6.9541 & 330.00 & 2 & 300 & 6.8379 \\
\bottomrule
\end{tabular}
\end{table}

\subsection{Complexity with respect to sparsity}
\label{subsec:sparsity-scaling}

Finally, we test the dependence of the MF-LogDet runtime on the sparsity scale$
ms,$
where $m$ is the number of monomials and $s$ is the average support size.
We fix the ambient dimension at
$
d=200
$
and vary the number of monomials
\[
m\in\{200,400,800,1200,2000,3000,4000\},
\]
while keeping the quartic support structure bounded.
In this way, the effective sparsity scale $ms$ varies substantially while the ambient dimension remains fixed.
As in the previous complexity experiment, we use $q=2$, PCG with maximum budget $k=5$, and regularization parameter $\lambda=10^{-6}$.

\paragraph{Results.}
Figure~\ref{fig:mf-logdet-sparsity} shows the measured runtime as a function of $ms$.
As $m$ increases from $200$ to $4000$, the effective sparsity scale $ms$ increases from $1000$ to $16200$, while the runtime per direction evaluation grows from $0.2181$ seconds to $3.0217$ seconds.
A log-log fit yields
\[
T \approx C (ms)^\beta,
\qquad
\beta = 0.9625.
\]
Since the fitted exponent is very close to $1$, the experiment provides strong numerical evidence that the runtime of MF-LogDet grows approximately linearly with $ms$ in the sparse quartic regime.

\paragraph{Interpretation.}
This experiment complements the previous dimension-scaling study. There, we observed near-linear growth with respect to $d$ under the regime $m=\Theta(d)$ and bounded support size. Here, by fixing $d$ and varying the sparsity level directly, we isolate the dependence on the structural quantity $ms$ itself. The resulting slope $\beta\approx 0.96$ is fully consistent with the complexity model of Section~\ref{sec:complexity}, which predicts that the matrix-free polynomial implementation is governed primarily by the algebraic sparsity of the objective rather than by the ambient dimension alone.

\begin{figure}[t]
    \centering
    \includegraphics[width=0.9\textwidth]{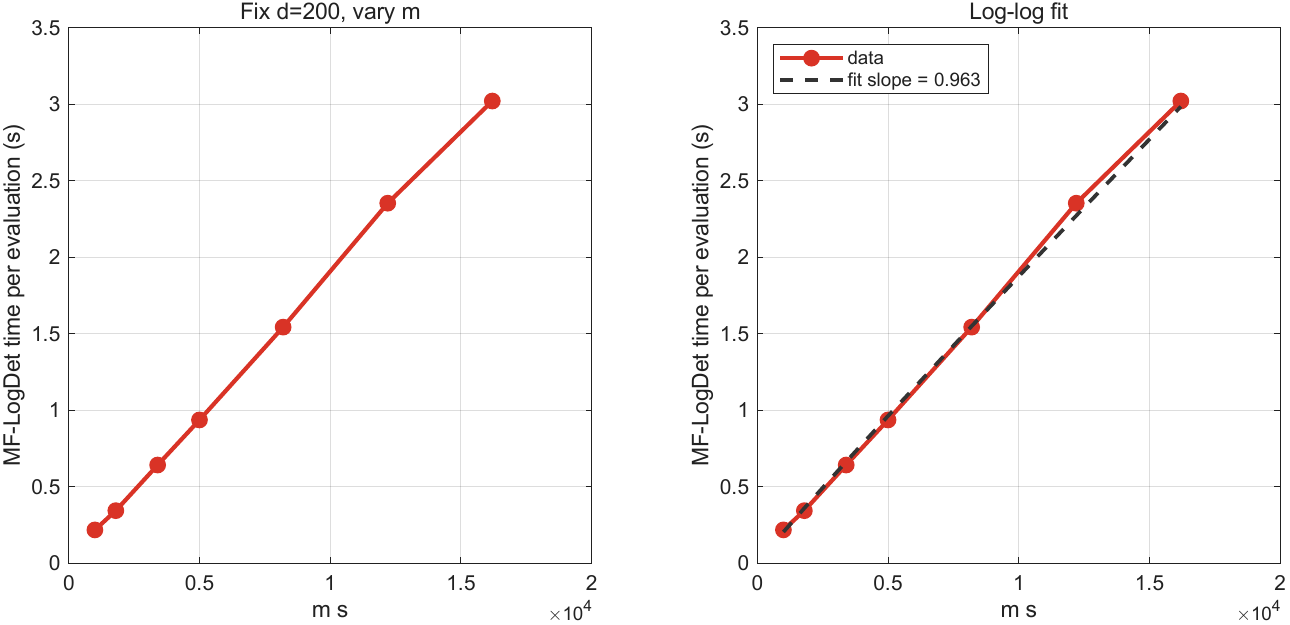}
    \caption{Runtime of one MF-LogDet affine normal evaluation versus the sparsity scale $ms$ for sparse quartic polynomials with fixed dimension $d=200$ and varying monomial count $m$.
The log-log fitted slope is approximately $0.9625$, providing empirical support for near-linear scaling with respect to the structural sparsity scale.}
    \label{fig:mf-logdet-sparsity}
\end{figure}

\begin{table}[t]
\centering
\caption{Empirical sparsity scaling of MF-LogDet at fixed dimension $d=200$. The runtime is measured for one affine normal direction evaluation.}
\label{tab:mf-logdet-sparsity}
\begin{tabular}{ccccc}
\toprule
$m$ & Avg. support & $ms$ & Time / eval (s) & Log-log fit \\
\midrule
200  & 2.5000 & 1000  & 0.2181 & 0.2046 \\
400  & 3.0000 & 1800  & 0.3440 & 0.3602 \\
800  & 3.4000 & 3400  & 0.6424 & 0.6644 \\
1200 & 3.5714 & 5000  & 0.9369 & 0.9630 \\
2000 & 3.7273 & 8200  & 1.5438 & 1.5503 \\
3000 & 3.8125 & 12200 & 2.3537 & 2.2724 \\
4000 & 3.8571 & 16200 & 3.0217 & 2.9856 \\
\bottomrule
\end{tabular}
\end{table}

\subsection{Discussion}

The numerical experiments support the main computational claims of this work.

\paragraph{(i) Exactness of the reformulation.}
In the exact/direct comparison with the original autodifferentiation-based computation, MF-LogDet reproduces the affine normal direction to near machine precision. This confirms that the log-determinant reformulation preserves the affine normal geometry while avoiding explicit third-order tensor contraction.

\paragraph{(ii) Accuracy--cost tradeoff of stochastic trace estimation.}
When the tangent log-determinant gradient is computed by stochastic trace estimation, the resulting affine normal direction remains accurate for moderate probe counts, while the computational cost grows approximately linearly with the number of probes. The runtime comparison with exact tangent log-determinant evaluation further shows that the stochastic approximation becomes increasingly attractive as the problem size grows.

\paragraph{(iii) Structural speedup and scalability.}
The observed acceleration is not merely a constant-factor improvement, but reflects a genuine transition from explicit tensor-based computation to matrix-free operator-based computation. In the sparse polynomial regime, the method exhibits near-linear scaling both with respect to the ambient dimension $d$ and with respect to the structural sparsity scale $ms$, in strong agreement with the theoretical complexity analysis.

\subsection{Summary}

Taken together, these experiments show that the proposed log-determinant reformulation provides an accurate and scalable route to affine normal computation. The exact version reproduces the original affine normal direction with essentially no loss of accuracy, while the stochastic matrix-free implementation offers a practical accuracy--efficiency tradeoff and becomes increasingly favorable as the problem dimension and sparsity grow. This demonstrates that the main computational bottleneck of explicit affine normal evaluation can be effectively removed, making affine normal computation viable in moderately large and high-dimensional sparse polynomial settings.

\section{Conclusion}
\label{sec:conclusion}

In this paper, we developed a scalable computational framework for affine normal directions in high-dimensional optimization, with particular emphasis on sparse polynomial objectives. The central contribution is a structural reduction showing that the key affine-normal contraction can be rewritten as the gradient of the log-determinant of the tangent Hessian. At the mathematical level, this reveals a hidden log-determinant geometry behind the classical affine normal formula. At the computational level, it replaces explicit third-order tensor contraction by a matrix-free formulation based on tangent Hessian operators, linear system solves, and stochastic trace estimation.

Building on this reformulation, we developed both exact and stochastic procedures for computing affine normal directions. In the polynomial setting, the algebraic closure of derivatives makes it possible to implement all required primitives through sparse monomial-based kernels. The resulting theory shows that the cost of affine normal computation is governed by matrix-free second-order primitives together with the stochastic and Krylov budgets, and is near-linear in the relevant sparsity scale. The numerical experiments confirm that the reformulation reproduces the affine normal direction computed by the explicit autodifferentiation-based formula to near machine precision, while delivering substantial runtime improvements in moderate and high dimensions and exhibiting near-linear empirical scaling in both dimension and sparsity.

More broadly, this work suggests that affine normal directions may be understood not only as classical third-order geometric objects, but also as curvature-corrected second-order search directions that incorporate local volume distortion through $\log\det(H_T)$. Several directions remain open for future study, including the design of effective preconditioners for tangent linear systems, multi-CPU and GPU implementations of the matrix-free kernels, and extensions beyond the polynomial setting to broader structured function classes such as trigonometric polynomials, symmetric polynomials, and low-rank structured models. An especially important next step is to integrate the present computational framework into large-scale implementations of Yau's affine normal descent. These directions appear promising for further connecting affine differential geometry with modern large-scale optimization.

\section*{Acknowledgements}
Y.-S.\ N.\ gratefully acknowledges support from the National Natural
Science Foundation of China (Grant No.\ 42450242) and the Beijing Overseas
High-Level Talent Program. A.\ S.\ would like to acknowledge support from the Beijing Natural Science
Foundation (Grant No.\ BJNSF--IS24005) and the NSFC--RFIS Program
(Grant No.\ W2432008).
He also thanks the NSF AI Institute for Artificial Intelligence and Fundamental
Interactions at the Massachusetts Institute of Technology (MIT), funded by the
U.S.\ National Science Foundation under Cooperative Agreement PHY--2019786,
as well as China's National Program of Overseas High-Level Talent for generous
support.
All three authors gratefully acknowledge institutional support from the Beijing
Institute of Mathematical Sciences and Applications (BIMSA).

\bibliographystyle{plain}
\bibliography{references}

\section*{Author Information}

\noindent
Yi-Shuai Niu$^{1}$, Artan Sheshmani$^{1,3}$, and Shing-Tung Yau$^{1,2}$

\medskip

\noindent
$^{1}$ Beijing Institute of Mathematical Sciences and Applications (BIMSA), Beijing 101408, China

\noindent
$^{2}$ Yau Mathematical Sciences Center, Tsinghua University, Beijing 100084, China

\noindent
$^{3}$ IAIFI Institute, Massachusetts Institute of Technology, Cambridge, MA 02139, USA

\medskip

\noindent
\textit{E-mail addresses:}\\
niuyishuai@bimsa.cn (Yi-Shuai Niu),\\
artan@mit.edu (Artan Sheshmani),\\
styau@tsinghua.edu.cn (Shing-Tung Yau)

\end{document}